\documentclass[10pt]{article}

\usepackage[OT2,T1]{fontenc}
\usepackage[russian,english]{babel}

\usepackage{amsmath,amsfonts,amssymb,amsthm}
\usepackage{enumerate}

\usepackage{a4}
\advance\textwidth-4pt\relax
\advance\oddsidemargin+2pt\relax
\advance\evensidemargin+2pt\relax

\makeatletter
\renewcommand\section{\@startsection {section}{1}{\z@}%
                                   {-3.5ex \@plus -1ex \@minus -.2ex}%
                                   {2.3ex \@plus.2ex}%
                                   {\normalfont\large\bfseries}}
\renewcommand\subsection{\@startsection{subsection}{2}{\z@}%
                                     {-3.25ex\@plus -1ex \@minus -.2ex}%
                                     {1.5ex \@plus .2ex}%
                                     {\normalfont\normalsize\bfseries}}

\newtheoremstyle{styleDef}{10pt}{10pt}{\upshape}{}{\scshape}{.}{7pt}{\thmnumber{{\upshape\bfseries#2.}\ }\thmname{#1}\thmnote{ #3}}
\newtheoremstyle{styleThm}{10pt}{10pt}{\slshape}{}{\scshape}{.}{7pt}{\thmnumber{{\upshape\bfseries#2.}\ }\thmname{#1}\thmnote{ #3}}

\theoremstyle{styleDef}
  \newtheorem{Definition}{Definition}
\theoremstyle{styleThm}
 \newtheorem{Theorem}[Definition]{Theorem}
 \newtheorem{Proposition}[Definition]{Proposition}
 \newtheorem{Lemma}[Definition]{Lemma}
 \newtheorem{Corollary}[Definition]{Corollary}

\let\le\leqslant  % less or equal
\let\ge\geqslant  % greater or equal

\DeclareMathOperator\tr{tr}      % trace
\DeclareMathOperator\ad{ad}      % adjoint representation
\DeclareMathOperator\Span{span}  % spanned vector space

\newcommand*\id{\mathrm{id}}              % identity

\newcommand*\ave[1]{\langle #1 \rangle}   % average
\newcommand*\scal[2]{\langle #1\mid\nobreak #2\rangle} % scalar product
\newcommand*\abs[1]{\left|#1\right|}      % absolute value
\newcommand*\norm[1]{\left\|#1\right\|}   % norm
\newcommand*\card[1]{\left|#1\right|}     % cardinal of a set

\newcommand*\ZZ{\mathbb{Z}}      % set of integers
\newcommand*\RR{\mathbb{R}}      % set of real numbers
\newcommand*\CC{\mathbb{C}}      % set of complex numbers
\newcommand*\PP{\mathbb{P}}      % projective space
\newcommand*\Se{\mathbb{S}}      % sphere

\DeclareMathOperator\Mat{Mat}    % matrices
\DeclareMathOperator\End{End}    % linear endomorphisms

\newcommand*\Orth{\mathbf{O}}    % orthogonal group
\newcommand*\Sp{\mathbf{Sp}}     % symplectic group
\newcommand*\SL{\mathbf{SL}}     % special linear group
\newcommand*\GL{\mathbf{GL}}     % general linear group

\newcommand*\so{\mathfrak{so}}   % special orthogonal Lie algebra
\newcommand*\slin{\mathfrak{sl}} % special linear Lie algebra
\newcommand*\glin{\mathfrak{gl}} % general linear Lie algebra

\newcommand*\G{\mathbf{G}}       % group
\newcommand*\tildeG{\mathbf{\tilde G}}
\newcommand*\GK{\mathbf{K}}      % orthogonal subgroup
\newcommand*\GP{\mathbf{P}}      % positive definite part of group
\newcommand*\GD{\mathbf{D}}      % triangular group
\newcommand*\GZ{\mathbf{Z}}      % center of group

\newcommand*\g{\mathfrak{g}}     % Lie algebra
\newcommand*\gk{\mathfrak{k}}    % orthogonal Lie subalgebra
\newcommand*\gp{\mathfrak{p}}    % selfdual part of Lie algebra
\newcommand*\gz{\mathfrak{z}}    % center of Lie algebra

\newcommand*\pp{\mathfrak{p}}    % graded subspaces of associative algebra
\newcommand*\Fc{\mathcal{F}}     % space of functions

\newcommand*\Q{\mathcal{Q}}
\newcommand*\Z{\mathcal{Z}}

\newcommand*\pf[1]{${#1}$\nobreakdash}
\newcommand*\smz{\setminus\{0\}}

\newcommand*\rbreak{\penalty\relpenalty}
\newcommand*\quartet[4]{(#1,\rbreak#2,\rbreak#3,\rbreak#4)}
\newcommand*\Quartet[4]{\bigl(#1,\rbreak#2,\rbreak#3,\rbreak#4\bigr)}

\begin{document}

\title{Extremality and designs in spaces of quadratic forms}
\author{Claude \textsc{Pache}%
  \footnote{%
    The author acknowledges support from the \emph{Swiss National Science Foundation}.\endgraf
    The author thanks the Institut de Math\'ematiques of the Universit\'e de Bordeaux~1 (France)
    for its hospitality during the preparation of this work.\endgraf
    Claude Pache,
    Universit\'e de Gen\`eve, Section de Math\'ematiques,
    C.P.~64, 1211 Gen\`eve 4, Switzerland.
    Email: Claude.Pache@gmail.com\endgraf
    MSC 2000: 11H55 (05B30, 11M41)\endgraf
    Keywords: quadratic form; extreme form; design; Epstein zeta function
    }\\
  }
\date{19 May 2008}

\maketitle

\begin{abstract}
A well known theorem of Voronoi
characterizes extreme quadratic forms and Euclidean lattices,
that is those which are local maxima for the Hermite function,
as perfect and eutactic.
This characterization has been extended in various cases,
such that family of lattices,
sections of lattices,
Humbert forms, etc.
Moreover, there is a criterion for extreme lattices, discovered by Venkov,
formulated in terms of spherical designs
which has been extended
in the case of Grassmannians and
sections of lattices.

In this article, we define a general frame,
in which there is a ``Voronoi characterization'',
and a ``Venkov criterion''
through an appropriate notion of design.
This frame encompasses many interesting situations
in which a ``Voronoi characterization'' has been proved.

We also discuss the question of
extremality relatively to the Epstein zeta function,
and we extend to our frame 
a characterization of final zeta-extremality
formulated by Delone and Ryshkov
and a criterion in terms of designs
found by Coulangeon.
\end{abstract}

Consider the following function on positive definite quadratic forms
on $\RR^n$ of determinant~1,
\begin{equation*}
  \gamma_n(Q) = \min_{z\in\ZZ^n\smz} Q(z),
\end{equation*}
which is called the \emph{Hermite function},
and its maximal value, $\gamma_n$,
which is the \emph{Hermite constant}.
The problem of estimating $\gamma_n$ is an important subject
of research,
and the systematic investigation of \emph{extreme} forms---%
the forms which are local maxima for the Hermite function---%
comes back to Korkine and Zolotarev in the nineteenth century.

The problem has an equivalent formulation in terms of Euclidean lattices:
For a lattice $\Lambda\subseteq\RR^n$ of determinant~$1$, we define:
\begin{equation*}
  \gamma_n(\Lambda) = \min_{z\in\Lambda\smz} \scal{z}{z}.
\end{equation*}
The Hermite constant $\gamma_n$ is equal to the maximal value
of this function among all Euclidean lattices of dimension~$n$
and determinant~$1$.
Both formulations have their own advantages and drawbacks;
the latter one provides a more graphical representation
of the problem and is naturally related to
the problem of finding densest ball packings in Euclidean space.

A famous theorem of Voronoi (1908) \cite{Voronoi} characterizes
extreme forms as perfect and eutactic
(the definition of these notions is given later).
This theorem has been extended in various situation,
as dual-extremality of lattices
(Berg\'e, Martinet, 1989 \cite{DualExtreme}),
extremality in families of lattices
(Berg\'e, Martinet, 1991 \cite{ExtremeAutomorphism}, \cite{LatticeFamily}),
extremality for the Rankin function of lattices
(Coulangeon, 1996 \cite{CoulangeonkExtreme}),
extremality of Humbert forms
(Coulangeon, 2001 \cite{CoulangeonVenkov}),
and extremality in a more general frame,
which includes systoles on a Riemannian manifold
(Bavard, 1997 \cite[Theorem~2.1]{Bavard}).

A relation between extreme lattices and so-called \emph{spherical designs}
has been discovered by Boris Venkov (1998), who has shown that a lattice
whose minimal layer forms a \pf{4}-design is perfect and eutactic
\cite{VenkovMartinet1}.
Recall that the \emph{minimal layer} of a Euclidean lattice $\Lambda$
is the set of vectors $x$ of~$\Lambda$ such that $\scal{x}{x}$
is minimal among all nonzero vectors of $\Lambda$,
and that
a \emph{spherical \pf{\tau}-design}
is a finite subset $X$ of the Euclidean sphere $\Se^{n-1}$ such that
\begin{equation*}
  \frac{1}{\card{X}}\sum_{x\in X} f(x) 
  = \frac{1}{\mathrm{vol}(\Se^{n-1})}\int_{\Se^{n-1}} f(x)
\end{equation*}
for all polynomial functions $f:\RR^n\to\RR$ of degree at most~$\tau$.
The criterion of Venkov has been extended to
extremality relatively to the Rankin function
through an appropriate notion of design in Grassmannian spaces
(Bachoc, Coulangeon, Nebe, 2002 \cite{DesignGrassmannian}).

There is another kind of extremality for quadratic forms and lattices,
namely the extremality of the so-called Epstein zeta function :
To a quadratic form $Q$ on $\RR^n$, and a complex number $s$, we associate
the series
\begin{equation*}
  \zeta(Q,s) = \sum_{x\in\ZZ^n\smz} Q(x)^{-s},
\end{equation*}
which converges for $\Re s> n/2$. The corresponding series for a lattice
$\Lambda\subseteq\RR^n$ is
\begin{equation*}
  \zeta(\Lambda,s) = \sum_{x\in\Lambda\smz} \scal{x}{x}^{-s}.
\end{equation*}
A quadratic form or a lattice is \emph{\pf{\zeta}-extreme} at $s>n/2$
when the value of its zeta function at~$s$ is locally minimal
among quadratic forms or lattices of same determinant.
A characterization of finally \pf{\zeta}-extreme lattices---%
that is lattices which are \pf{\zeta}-extreme at~$s$ for all
$s$ large enough---has been formulated by
Delone and Ryshkov (1967) \cite{DeloneRyshkov}: a Euclidean lattice
is finally \pf{\zeta}-extreme if and only if it is perfect and all its layers
are strongly eutactic. This result may be seen as
a version for final \pf{\zeta}-extremality
of Voronoi's characterization of extremality.
Such a criterion has also been recently proved
for Humbert forms by Coulangeon~\cite{CoulangeonEpsteinHumbert}.

As for classic extremality, there is a criterion for \pf{\zeta}-extremality
in terms of designs, discovered by Coulangeon (2006) \cite{CoulangeonEpstein}:
if all layers of a lattice of dimension~$n$ are spherical \pf{4}-designs,
then it is \pf{\zeta}-extreme for all $s>n/2$.

\medskip

The goal of this paper is to provide a theoretical frame in which the four
results on extremality mentioned above hold.
Our paper is organized as follows:

\begin{itemize}
  \item In Section~\ref{sectVoronoiSpaces},
    we give the definitions and main properties
    of our frame, that we call \emph{Voronoi space},
    and we give examples of Voronoi spaces,
    showing how to integrate in our frame
    most of the cases mentioned in the introduction;
  \item In Section~\ref{sectGroupVoronoiSpace},
    we use some theory of linear Lie groups
    in order to establish the structure of the group
    of a Voronoi space, which is a fundamental object in
    our theory;
  \item In Section~\ref{sectCharacterizationVoronoi},
    we prove a characterization \`a la Voronoi
    of extremality and strict extremality
    in Voronoi spaces
    (Theorem~\ref{thmVoronoi}).
    We also establish a property which, in many cases,
    may be used to show easily the
    equivalence of extremality and strict extremality
    (Proposition~\ref{propWeaklyPerfectOnly});
  \item In Section~\ref{sectAQV}, we introduce
    \emph{alternative Voronoi spaces}, which are essentially equivalent
    to Voronoi spaces with respect to our problem of extremality,
    just as quadratic forms are equivalent to lattices with respect to the
    Hermite function;
  \item In Section~\ref{sectDesign}, we define a notion of design for appropriate
    alternative Voronoi spaces, and we prove
    the criterion corresponding to the criterion of Venkov
    (Theorem~\ref{thmVenkovCriterion}
    and Corollary~\ref{corPerfectionCriterion});
  \item Finally, the problem of \pf{\zeta}-extremality is introduced
    in Section~\ref{sectEpstein}, and the corresponding characterization
    of final extremality
    (Theorem~\ref{thmDeloneRyshkov})
    and the criterion using designs
    (Theorem~\ref{thmVenkovEpstein})
    are proved.
\end{itemize}

The theory developed in this paper is not the first attempt to provide
a unification of the problem of extremality introduced above.
In particular, the reader may check that our notion of Voronoi space
is a particular case of a more general frame introduced by Christophe Bavard,
in which there is also a ``Voronoi characterization'' under a
condition of convexity, called ``Condition~C'' \cite[\S2.2]{Bavard}.
In our frame, we do not need such a condition,
but instead we use the structure of the Lie group acting on the
situation to obtain the convexity property
given by Lemma~\ref{lemEstimationNeighborhood}
and the reducibility criterion given by Proposition~\ref{propWeaklyPerfectOnly}.
Moreover, there is no natural notion of design in the frame of Bavard.
On the other hand, our frame does not encompass
all interesting cases;
in particular, we have not succeeded to apply it
to the problem of systoles on Riemannian manifolds,
that Bavard has studied in the cited paper.

Note also that the characterization of \pf{\zeta}-extremality
of Delone and Ryshkov
(Theorem~\ref{thmDeloneRyshkov}) may also be proved in
the frame of Bavard.

\section{Voronoi spaces}\label{sectVoronoiSpaces}

As stated in the introduction, the goal of this paper
is to provide a practical frame in which there is
an equivalent of the
characterization of extremality of
Voronoi~\cite{Voronoi} and some other results on extremality
that have been extended in various cases.
The notion of \emph{Voronoi space} we introduce here
is based on the observation of
features which often appear in the frames where those different
results of extremality have been proved.
There are: a ground space that can
be expressed as real submanifold of the cone of positive definite
quadratic forms on a
finite-dimensional real vector space,
and the high symmetry of that ground space
which is expressed by the existence
of a Lie group acting naturally and transitively on it.

In the last paragraph of this section (Paragraph~\ref{parExamples}),
we give some examples of Voronoi spaces
corresponding to situations studied before by others.

\subsection{General conventions and notations}
Let $V$ be a finite-dimensional real vector space.
If $Q$ is a quadratic form on~$V$, we denote by the same symbol $Q$
the associated symmetric bilinear form. The relation between the two is:
\begin{gather*}
  Q(x,y) = \frac12 \Bigl( Q(x+y)-Q(x)-Q(y) \Bigr),\\
  Q(x) = Q(x,x).
\end{gather*}
For a nondegenerate quadratic form (or symmetric bilinear form) $Q$ on $V$,
and a linear transformation~$H$ of $V$, we denote by $H^*_Q$
the \emph{\pf{Q}-adjoint} of $H$,
that is the unique linear transformation of $V$ such that
$Q(x,Hy)=Q(H^*_Qx,y)$ for every $x,y\in V$.
The linear transformation~$H$ is \emph{\pf{Q}-selfadjoint},
respectively \emph{\pf{Q}-antiselfadjoint},
respectively \emph{\pf{Q}-orthogonal},
when $H^*_Q=H$, respectively $H^*_Q=-H$, respectively $H^*_Q=H^{-1}$.

Let $Q$ be a positive definite quadratic form on $V$.
We define the norm $\norm{\cdot}_{Q}$ on the associative algebra~$\End(V)$
of linear endomorphisms of~$V$ by:
\begin{equation*}
  \norm{H}_Q^2 = \sup_{x\in V\smz}\frac{Q(Hx)}{Q(x)}, \quad H\in\End(V).
\end{equation*}

If $x$ is an element or a subset in $V\smz$,
we denote by $[x]$ the corresponding element or subset in the projective space
$\PP(V)$, that is the image of $x$ through the projection $V\smz\to\PP(V)$.

\subsection{Voronoi space}
Here is the frame in which we do our theory:

\begin{Definition}\label{defVoronoiSpace}
  A \emph{Voronoi space} $\quartet{V}{\Q}{Z}{\G}$ is the data of
  \begin{enumerate}[(a)]
    \item a finite-dimensional real vector space $V$,
    \item a connected submanifold $\Q$ in the half-cone
      of positive definite quadratic forms on $V$,
    \item a closed discrete subset $Z$ of $V$ not containing~$0$
      and spanning~$V$,
    \item a closed linear Lie group $\G<\GL(V)$,
  \end{enumerate}
  satisfying the following properties:
  \begin{enumerate}[(i)]
    \item For $Q\in\Q$ and $c>0$, we have $cQ\in\Q$ only if $c=1$;
    \item for all $g\in\G$ and $Q\in\Q$, we have $g^*_Q\in\G$;
    \item for all $g\in\G$ and $Q\in\Q$, we have $Q\circ g\in\Q$,
      and $\G$ is transitive on $\Q$.
  \end{enumerate}
\end{Definition}

The point of restriction~(i)
is to avoid to deal with some special cases which would 
be anyway trivially uninteresting for the problems
of extremality which occupy us in this paper.

The standard example of Voronoi space is $V=\RR^n$ with $n\ge2$,
and $\Q$ the set of positive definite quadratic forms of determinant~1,
together with $Z=\ZZ^n\smz$ and $\G=\SL^{\pm}(n,\RR)$, where
\begin{equation}
  \SL^{\pm}(n,\RR) = \{ g\in \GL(n,\RR) \mid \det g =\pm1 \}.
\end{equation}
We refer to this example as the ``classic Voronoi space
of rank~$n$''.
Before we give other examples, we fix some general notations and properties
of Voronoi spaces.

The group $\G$ of the definition is the \emph{group} of the Voronoi space.
Given a Voronoi space $\quartet{V}{\Q}{Z}{\G}$, there exist two linear groups $\G_{\min}$
and $\G_{\max}$ such that $\quartet{V}{\Q}{Z}{\tildeG}$ is a Voronoi space
if and only if we have $\G_{\min}<\tildeG<\G_{\max}$.
Moreover, $\quartet{V}{\Q}{Z}{\mathbf{\tilde G}}$ is equivalent to $\quartet{V}{\Q}{Z}{\G}$
with respect to the problems of extremality
that occupies us in this paper.
For the classic Voronoi space of rank~$n$,
we have $\G_{\min}=\SL(n,\RR)$ and $\G_{\max}=\SL^{\pm}(n,\RR)$.

\subsection{Group of a Voronoi space}
The group of a Voronoi space and the related objects
play an important role in our theory,
as they assure the ``convexity'' of the functions
$Q\mapsto Q(x)$, on $\Q$ (Lemma~\ref{lemEstimationNeighborhood}).

\begin{Definition}\label{defGroup}
  Let $\quartet{V}{\Q}{Z}{\G}$ be a Voronoi space,
  and let $Q\in\Q$.
  Then $\G$ is the \emph{group} of the Voronoi space.
  Moreover, we define the following objects:
  \begin{itemize}
    \item $\GK_Q=\{g\in\G \mid g^*_Q=g^{-1}\}$, the (compact) group
      of \pf{Q}-selfadjoint elements of~$\G$;
    \item $\g$, the Lie algebra of~$\G$;
    \item $\gk_Q = \{H\in\g\mid H+H^*_Q=0\}$, the Lie subalgebra of
      \pf{Q}-antiselfadjoint elements of~$\g$,
      which is the Lie algebra of~$\GK_Q$;
    \item $\gp_Q = \{H\in\g\mid H-H^*_Q=0\}$, the set of \pf{Q}-selfadjoint
      elements of~$\g$;
    \item $\G^0$ and $\GK_Q^0$, the connected component of the identity
      of $\G$ and $\GK_Q$ respectively.
  \end{itemize}
Note that we have $\g=\gk_Q\oplus\gp_Q$,
and, for any $c\in\RR$, we have $c\,\id_V\in\G$ if and only if $c=\pm1$,
and $c\,\id_V\in\g$ if and only if $c=0$.
\end{Definition}

\begin{Lemma}\label{lemConjugate}
  Let $Q,Q'\in\Q$ and let $g\in\G$ such that $Q'=Q\circ g$. Then
  \begin{equation*}
    \gp_{Q'} = g^{-1} \gp_{Q}g,
    \quad\text{and}\quad \gk_{Q'} = g^{-1} \gk_{Q}g.
  \end{equation*}
\end{Lemma}
\begin{proof}
  Let $H\in\g$.
  We have, for any linear endomorphism~$A$ of~$V$,
  \begin{equation*}
    A^*_{Q'} = (g^*_Qg)^{-1} A^*_{Q} (g^*_Qg);
  \end{equation*}
  in particular, setting $A=g^{-1}Hg$, we get
  \begin{equation*}
    (g^{-1}Hg)^*_{Q'}
    = (g^*_Qg)^{-1} (g^{-1}Hg)^*_{Q} (g^*_Qg)
    = g^{-1} H^*_{Q}g.
  \end{equation*}
  Therefore $H$ is \pf{Q}-selfadjoint, respectively \pf{Q}-antiselfadjoint,
  if and only if $g^{-1}Hg$ is \pf{Q'}-selfadjoint,
  respectively \pf{Q'}-antiselfadjoint.
\end{proof}

\subsection{Layers, Hermite function and extremality}

\begin{Definition}\label{defLayerQV}
  Let $\quartet{V}{\Q}{Z}{\G}$ be a Voronoi space.
  \begin{enumerate}[(i)]
    \item The \emph{Hermite function} associated to $\quartet{V}{\Q}{Z}{\G}$ is the
    function $\gamma$ on $\Q$ defined by
      \begin{equation*}
        \gamma(Q) = \min_{x\in Z} Q(x),\quad Q\in\Q.
      \end{equation*}
    \item For $r>0$, the \emph{layer} of square radius~$r$ of a quadratic form $Q\in\Q$ is the set
      \begin{equation*}
        Q_{r} = \{x\in Z \mid Q(x)=r \}.
      \end{equation*}
    \item The set of \emph{minimal vectors} of a quadratic form $Q\in\Q$
      is the nonempty layer of minimal radius, that is
      \begin{equation*}
        Q_{\min} = Q_{\gamma(Q)} = \{x\in Z \mid Q(x)=\gamma(Q) \}.
      \end{equation*}
    \item A quadratic form $Q\in\Q$ is \emph{extreme}, respectively \emph{strictly extreme},
      when it is a local maximum, respectively a strict local maximum, of~$\gamma$.
  \end{enumerate}
\end{Definition}
In the case of the classic Voronoi space of rank~$n$,
the function~$\gamma$ is the ``classic'' Hermite function,
and its maximal value is the Hermite constant $\gamma_n$.

The main goal of this paper is to find criteria for extreme forms
in Voronoi spaces.
The first question to consider is
if the supremum of $\gamma$
 on $\Q$ is indeed a local maximum.
 Sadly, we have not found any easy criterion to check this property;
 nevertheless, the following definition is justified
 by many interesting examples of Voronoi spaces.

Let $\quartet{V}{\Q}{Z}{\G}$  be a Voronoi space,
and consider the discrete subgroup
\begin{equation*}
  \mathbf{H} = \{g\in\G \mid gZ=Z\}.
\end{equation*}
The quotient $\Q/\mathbf{H}$ is separated, and we have a map
\begin{align*}
  \bar{\gamma}:\Q/\mathbf{H} & \longrightarrow ]0,\infty[ \\
   Q\mathbf{H} & \longmapsto \gamma(Q),
\end{align*}

\begin{Definition}\label{defCompact}
  The Voronoi space $\quartet{V}{\Q}{Z}{\G}$ is \emph{compact}
  when $\bar{\gamma}$
  is bounded and proper.
\end{Definition}

In a compact Voronoi space,
the supremum of the Hermite function is a also a local maximum.

The statement that classic Voronoi spaces are compact
is essentially equivalent to Mahler's Compactness Theorem
(see, e.g., \cite[\S2.4]{MartinetBook}).

\subsection{Examples}\label{parExamples}

The examples of Voronoi space we give here
show how to integrate most cases cited in the introduction
in our frame.
All these examples have an alternative formulation,
which is made explicit in Section~\ref{sectAQV}.

\begin{enumerate}[1.]
  \item\label{exClassical} \textit{Classic spaces.}
    Let $V=\RR^n$ with $n\ge2$,
    let $\Q$ be the set of positive definite quadratic forms on $\RR^n$
    of determinant~$1$, let $Z=\ZZ^n\smz$,
    and let $\G=\SL^{\pm}(n,\RR) = \{g\in\GL(n,\nobreak\RR) \mid \det g=\pm1 \}$.
    Then $\quartet{V}{\Q}{Z}{\G}$ is a compact Voronoi space with
    the following parameters, where $Q_0$ is the canonical quadratic form
    on $\RR^n$:
    \begin{gather*}
      \GK_{Q_0} = \{g\in\Mat(n,\RR) \mid g^*=g^{-1}\} = \Orth(n), \\
      \g = \{H\in\Mat(n,\RR) \mid \tr H=0\} = \slin(n,\RR),\\
      \gk_{Q_0} = \{H\in\Mat(n,\RR) \mid H+H^*=0\ \text{and}\ \tr H=0\} = \so(n,\RR),\\
      \gp_{Q_0} = \{H\in\Mat(n,\RR) \mid H^*=H\ \text{and}\ \tr H=0\}.
    \end{gather*}
    For this space, the Hermite function is the classic Hermite function
    (restricted to quadratic forms of determinant~$1$),
    and its maximal value is the classic Hermite constant $\gamma_n$.
    
  \item\label{exDuality} \textit{Duality.}
    Let $\quartet{V}{\Q}{Z}{\G}$ be a Voronoi space.
    Let $Q\mapsto Q^\alpha$ be an involution of~$\Q$,
    and $g\mapsto g^\alpha$ an involution and automorphism of~$\G$
    such that $(g^*_Q)^\alpha = (g^\alpha)^*_{Q^\alpha}$ and
    \begin{equation*}
      (Q\circ g)^\alpha = Q^\alpha\circ g^\alpha,\quad
      \forall g\in G,\ \forall Q\in\Q.
    \end{equation*}
    Then we define a Voronoi space
    $\quartet{\tilde{V}}{\tilde{\Q}}{\tilde{Z}}{\tildeG}$ as follows: 
    \begin{gather*}
      \tilde{V} = V\times V \\
      \tilde{\Q} = \bigl\{\tilde{Q}
        \bigm| \tilde{Q}\bigl((x_1,x_2)\bigr) =Q(cx_1)+Q^\alpha(c^{-1}x_2),
        \text{for some $Q\in\Q$ and $c\in\RR^*$} \bigr\}, \\
      \tilde{Z} = \bigl(Z\times\{0\}\bigr) \cup \bigl(\{0\}\times Z\bigr),\\
      \begin{split}
        \tildeG = \Bigl\{ \tilde{g} \in\GL(V\times V)
          &\Bigm| \tilde{g} =
            \begin{pmatrix} cg & 0 \\ 0 & c^{-1}g^\alpha \end{pmatrix}\ \text{or}\
            \begin{pmatrix} 0 & cg \\ c^{-1}g^\alpha & 0 \end{pmatrix}\\[-2\jot]
          &\qquad\qquad\qquad\text{for some $g\in\G$ and $c\in\RR^*$} \Bigr\}.
      \end{split}
    \end{gather*}
    For $Q\in\Q$ and $c\in\RR^*$ let $\tilde{Q}_c\in\tilde{\Q}$
    such that
    $\tilde{Q}_c((x_1,\nobreak x_2))=Q(cx_1)+Q^\alpha(c^{-1}x_2)$.
    The Hermite function $\tilde{\gamma}$
    of $\quartet{\tilde{V}}{\tilde{\Q}}{\tilde{Z}}{\tildeG}$
    is equal to
    \begin{equation*}
      \tilde\gamma(\tilde{Q}_c)
      = \min\bigl(c^2\gamma(Q), c^{-2}\gamma(Q^\alpha)\bigr).
    \end{equation*}
    For a given $Q\in\Q$, the maximum of $\gamma(\tilde{Q}_c)$
    among all values of $c$ is attained when $c^2\gamma(Q) = c^{-2}\gamma(Q^\alpha)$,
    that is when $c^2=\gamma(Q)^{-1/2}\*\,\gamma(Q^\alpha)^{1/2}$,
    and, for this value of~$c$, we have
    \begin{equation*}
      \tilde\gamma(\tilde{Q}_c)= \bigl(\gamma(Q)\,\gamma(Q^\alpha)\bigr)^{1/2}.
    \end{equation*}
    Therefore, the local maxima of the Hermite function $\tilde\gamma$
    on $\tilde{\Q}$
    correspond to the maxima of the function
    $\gamma'(Q)= \bigl(\gamma(Q)\*\,\gamma(Q^\alpha)\bigr)^{1/2}$
    on $\Q$.
    A quadratic form $Q\in\Q$ such that the corresponding form
    $\tilde{Q}_c$ is extreme is called \emph{dual-extreme}.

    If $\quartet{V}{Q}{Z}{\G}$ is compact and
    if we have $g^\alpha Z=Z$ for every $g\in\G$ satisfying $gZ=Z$, then
    $\quartet{\tilde{V}}{\tilde{\Q}}{\tilde{Z}}{\tildeG}$ is compact.
  
    As a particular case, we take for $\quartet{V}{\Q}{Z}{\G}$ the classic Voronoi
    space of rank~$n$, with $g^\alpha=g^{{*}{-1}}$ and
    $Q^\alpha=Q^{-1}$, where $g^*$ denotes the adjoint of~$g$
    for the canonical scalar product on $V=\RR^n$, and where
    $Q^{-1}(x)=\scal{x}{A^{-1}x}$ if $Q(x)=\scal{x}{Ax}$.
    The maximal value $\gamma'_n$ of the Hermite function
    on the corresponding space $\quartet{\tilde{V}}{\tilde{\Q}}{\tilde{Z}}{\tildeG}$
    is known as \emph{Berg\'e-Martinet constant},
    and has been first studied in \cite{DualExtreme}.
    
  \item\label{exOrbit} \textit{Family of quadratic forms.}
    Let $\quartet{V}{\Q}{Z}{\G}$ be a Voronoi space,
    let $Q$ be an element of~$\Q$, 
    and let $\mathbf{H}$ be a connected closed subgroup of $\G$
    such that $g^*_Q\in\mathbf{H}$ when $g\in\mathbf{H}$.
    Let $Q\mathbf{H}$ be the orbit of $Q$ in $\Q$.
    Then $\quartet{V}{Q\mathbf{H}}{Z}{\mathbf{H}}$ is a Voronoi space,
    which is compact if $\quartet{V}{\Q}{Z}{\G}$ is compact.
    
    When $\quartet{V}{\Q}{Z}{\G}$ is a classic Voronoi space,
    the situation (but formulated in terms of lattices instead of
    quadratic forms, see Paragraph~\ref{parExamplesBis}) has been studied by
    Anne-Marie Berg\'e and Jacques Martinet
    in \cite{LatticeFamily}.
    As special cases, we have:
    \begin{enumerate}[(a)]
      \item 
        Let $\mathbf{\Gamma}$ be a finite subgroup of $\G$
        such that $\gamma Z=Z$ for every $\gamma\in\mathbf{\Gamma}$.
        An element $Q\in\Q$ such that the elements of $\mathbf{\Gamma}$
        are \pf{Q}-orthogonal is called
        \emph{\pf{\mathbf{\Gamma}}-invariant}.
        
        Let $Q\in\Q$ be a \pf{\mathbf{\Gamma}}-invariant quadratic form,
        and let $\mathbf{H}$ be the connected component of the identity
        of the centralizer of $\mathbf{\Gamma}$ in $\G$.
        The group $\mathbf{H}$ is \pf{Q}-selfadjoint,
        because so is $\mathbf{\Gamma}$. However, the following
        fact is less evident:
        \begin{Lemma}\label{lemExampleInvariant}
        $Q\mathbf{H}$ is the set of \pf{\mathbf{\Gamma}}-invariant elements of~$\Q$.
        \end{Lemma}
        \begin{proof}
        The nontrivial fact to show is that any
        \pf{\mathbf{\Gamma}}-invariant element of~$\Q$ is of the 
        form $Q\circ h$ with $h\in\mathbf{H}$.
        Let $Q'$ be such an element; according to Proposition~\ref{propAdjoinization},
        we have $Q'(x,y)=Q(x,\exp(H)y)$ for some $H\in\gp_Q$.
        Then, for $\gamma\in\mathbf{\Gamma}$, we have
        $Q'(\gamma x,\gamma y) = Q(x,\exp(\gamma^{-1} H\gamma)y)$;
        therefore the condition
        $Q'\circ\gamma=Q'$ is equivalent to
        \begin{equation*}
          \exp(\gamma^{-1}H\gamma) = \exp(H).
        \end{equation*}
        Since $\exp$ is injective on $\gp_Q$, the last condition is equivalent
        to $\gamma^{-1}H\gamma=H$. Then $h=\exp(H/2)$ is an element of $\mathbf{H}$,
        and we have $Q'=Q\circ h$.
        \end{proof}
        
        So, $\quartet{V}{Q\mathbf{H}}{Z}{\mathbf{H}}$ is a Voronoi space,
        and their extreme forms are local maxima
        of the Hermite function restricted
        to the \pf{\mathbf{\Gamma}}-invariant elements of~$\Q$.

        The case where $\quartet{V}{\Q}{Z}{\G}$ is a classic Voronoi space
        has been studied by Berg\'e and Martinet
        in \cite{ExtremeAutomorphism}.
        
      \item Let $Q\mapsto Q^\alpha$ and $g\mapsto g^\alpha$
        as in Example~\ref{exDuality}; let $\sigma\in\G$ such that
        $\sigma Z=Z$.
        An element $Q\in\Q$ such that $Q=Q^\alpha\circ\sigma$, is called 
        \emph{\pf{\sigma}-isodual}. Note that, if $Q\in\Q$ is \pf{\sigma}-isodual,
        we have $\gamma(Q)=\gamma(Q_\alpha)$.

        Let $Q\in\Q$ be a \pf{\sigma}-isodual quadratic form,
        Let $\mathbf{H}$ be the connected component of the identity
        of set of the elements $g\in\G$
        satisfying $\sigma g=g^\alpha\sigma$.
        \begin{Lemma}\label{lemExampleIsodual}
        $Q\mathbf{H}$ is the set of \pf{\sigma}-isodual elements of~$\Q$.
        \end{Lemma}
        \begin{proof}
        The general line of the proof is the same as the one of Lemma~\ref{lemExampleInvariant}:
        The nontrivial fact to show is that any \pf{\sigma}-isodual
        element of $\Q$ is of the form $Q\circ h$ with $h\in \mathbf{H}$.       
        Let $Q'$ be such an element; according to Proposition~\ref{propAdjoinization},
        we have $Q'(x,y)=Q(x,\exp(H)y)$ for some $H\in\gp_Q$.
        Let $H\mapsto H^A$ be the derivative of $g\mapsto g^\alpha$; we have
        $\exp(H^A)=\exp(H)^\alpha$ and $(H^*_Q)^A= (H^A)^*_{Q^\alpha}$.
        The property $(Q\circ h)^\alpha= Q^\alpha\circ h^\alpha$ applied
        to $h=\exp(H/2)$ gives
        $Q'^\alpha(x,y) = Q^\alpha\bigl(x,\exp(H^A)y\bigr)$.
        Therefore, we have
        \begin{multline*}
          Q\bigl(x,\exp(H) y\bigr)
          = Q'(x, y) 
          = Q'^{\alpha}(\sigma x, \sigma y) \\
          = Q^\alpha\bigl(\sigma x,\exp(H^A)\sigma y\bigr) 
          = Q^\alpha\bigl(\sigma x,\sigma \exp(\sigma^{-1}H^A\sigma) y\bigr) \\
          = Q\bigl(x,\exp(\sigma^{-1}H^A\sigma) y\bigr).
        \end{multline*}
        We infer the following equality:
        \begin{equation*}
          \exp(H) = \exp(\sigma^{-1}H^A\sigma).
        \end{equation*}
        If we have $\sigma^{-1}H^A\sigma\in\gp_Q$, the injectivity of $\exp$
        on $\gp_Q$ implies $H=\sigma^{-1}H^A\sigma$, and we have
        $Q'=Q\circ h$ with $h=\exp(H/2)\in\mathbf{H}$.
        Therefore, it suffices to show the equality
        \begin{equation*}
          (\sigma^{-1}H^A\sigma)^*_Q = \sigma^{-1}H^A\sigma.
        \end{equation*}       
        Since $Q^\alpha=Q\circ\sigma^{-1}$, we have,
        for any linear endomorphism $A$ of $V$, 
        \begin{equation*}
          A^*_{Q^\alpha}
          = \bigl((\sigma^{-1})^*_Q\sigma^{-1}\bigr)^{-1} A^*_Q \bigl((\sigma^{-1})^*_Q\sigma^{-1}\bigr)
          = (\sigma\sigma^*_Q) A^*_Q (\sigma\sigma^*_Q)^{-1}.
        \end{equation*}
        In particular, we get
        \begin{equation*}
          (H^A)^*_{Q} = (\sigma\sigma^*_Q)^{-1} (H^A)^*_{Q^\alpha} (\sigma\sigma^*_Q)
          =(\sigma\sigma^*_Q)^{-1} (H^*_Q)^A (\sigma\sigma^*_Q).
        \end{equation*}
        Therefore, we have
        \begin{equation*}
          (\sigma^{-1}H^A\sigma)^*_Q
          = \sigma^*_Q\,(\sigma\sigma^*_Q)^{-1} (H^*_Q)^A (\sigma\sigma^*_Q)\,(\sigma^{-1})^*_Q
          = \sigma^{-1}H^A\sigma,
        \end{equation*}
        which is what remained to prove.
        \end{proof}       
        \begin{Lemma}\label{lemExampleIsodualBis}
        For any $g\in\mathbf{H}$, and any $Q'\in Q\mathbf{H}$, we have $g^*_{Q'}\in\mathbf{H}$.
        \end{Lemma}
        \begin{proof}
        Let $h\in\mathbf{H}$, and let $Q'=Q\circ j$ with $j\in\mathbf{H}$. since we have
        $g^*_{Q'} = (j^*_Qj)^{-1} h^*_{Q} (j^*_Qj)$, it suffices to show the lemma for $Q'=Q$.
        By the same argument given in the previous proof for the calculation of $(H^A)^*_Q$, we have
        \begin{equation*}
          (h^\alpha)^*_Q = (\sigma\sigma^*_Q)^{-1} (h^*_Q)^\alpha (\sigma\sigma^*_Q),
        \end{equation*}
        and therefore
        \begin{equation*}
          \sigma h^*_Q
          = (\sigma^{-1}\sigma h \sigma^*_Q)^*_Q
          = (\sigma^{-1} h^\alpha\sigma \sigma^*_Q)^*_Q
          = (\sigma\sigma^*_Q) (h^\alpha)^*_Q (\sigma\sigma^*_Q)^{-1}\sigma
          = (h^*_Q)^\alpha\sigma,
        \end{equation*}
        therefore, we have $h^*_Q\in\mathbf{H}$.
        \end{proof}
        So,
        $\quartet{V}{Q\mathbf{H}}{Z}{\mathbf{H}}$ is a Voronoi space,
        and their extreme forms are local maxima
        of the Hermite function restricted
        to the \pf{\sigma}-isodual elements of~$\Q$.
        
        The dual case of this situation
        (see Paragraph~\ref{parExamplesBis}, Example~\ref{exOrbitAQV})
        with the classic Voronoi space,
        has been studied by Berg\'e and Martinet
        in \cite{LatticeFamily}.
        The case where $\sigma^2=-\id_V$
        is also interesting, see \cite{BuserSarnak}.
    \end{enumerate}   
  
  \item\label{exExterior} \textit{Exterior powers.}
    Let $\quartet{V}{\Q}{Z}{\G}$ be a Voronoi space and let $1\le m\le\dim(V)/2$.
    We define a space $\quartet{V^{\wedge m}}{\Q^{\wedge m}}{Z^{\wedge m}}{\G^{\wedge m}}$
    as follows: $V^{\wedge m}$ is the $m$th exterior power of $V$,
    and
    \begin{gather*}
      \Q^{\wedge m} = \{Q^{\wedge m}\mid Q\in\Q\},\
      \text{where}\ Q^{\wedge m}(x_1\wedge\dots\wedge x_m)
            = \det\bigl(Q(x_i,x_j)_{i,j=1}^{m}\bigr), \\
      Z^{\wedge m} = \{x_1\wedge\dots\wedge x_m\mid
        \text{$x_1,\dots,x_m$ are linearly independent elements of $Z$}\},\\        
      \G^{\wedge m} = \{\tilde{g} \mid g\in\G\},\
        \text{where $\tilde{g}(x_1\wedge\dots\wedge x_m) = gx_1\wedge\dots\wedge gx_m$.}
    \end{gather*}
    The elements of its Lie algebra $\g^{\wedge m}$ are of the form
    \begin{multline*}
      \tilde{H}(x_1\wedge\dots\wedge x_m) = 
      (Hx_1\wedge x_2\wedge\dots\wedge x_m) \\
      + (x_1\wedge Hx_2\wedge x_3\wedge\dots\wedge x_m)
      +\dots
      + (x_1\wedge\dots\wedge x_{m-1}\wedge Hx_m), \\
      \text{with  $H\in\g$}.
    \end{multline*}
    
    If $\quartet{V}{\Q}{Z}{\G}$ is the classic Voronoi space of rank~$n$,
    then $\quartet{V^{\wedge m}}{\Q^{\wedge m}}{Z^{\wedge m}}{\G^{\wedge m}}$ is a
    compact Voronoi space. The maximal value of the corresponding
    Hermite function is the maximal value of
    \begin{equation*}
      \det\Bigl(\bigl(\scal{x_i}{x_j}\bigr)_{i,j=1}^{m}\Bigr)
    \end{equation*}
    among all sets of $m$ linear independent \pf{n}-dimensional
    vectors $x_1,\dots,x_m$ with integral entries.
    It is known as the \emph{Rankin constant}~$\gamma_{n,m}$,
    and it has been first considered by Rankin in \cite{Rankin}.

  \item\label{exHumbert} \textit{Humbert forms.}
    Let $K$ be a number field of signature $(r,s)$
    of integral ring $\mathcal{O}_K$,
    let $\sigma_1,\dots,\sigma_r$ be the real embeddings of $K$,
    and $\sigma_{r+1},\dots,\sigma_{r+2s}$ be its complex embeddings,
    with $\sigma_{r+s+i}=\overline{\sigma_{r+i}}$,
    and let $n$ be a positive integer.
    In what follows, ``$\otimes$'' denotes always the tensor product as \emph{real} vector space.
    Let $V$ be the linear subspace of $(\RR^n)^{\otimes r}\otimes(\CC^n)^{\otimes 2s}$
    generated by $x_1\otimes\dots\otimes x_{r+2s}$, where
    \begin{gather*}
      x_i\in\RR^n \quad \text{for $1\le i\le r$},\\
      x_{r+i}\in\CC^n \ \text{and}\ x_{r+s+i}=\overline{x_{r+i}}\quad \text{for $1\le i\le s$}.
    \end{gather*}
    Let $Z$ be the subset of $V$ given by
    \begin{equation*}
      Z=\bigl\{\sigma_1(x)\otimes\dots\otimes\sigma_{r+s}(x)\bigm| x\in\mathcal{O}_K^n \bigr\}.
    \end{equation*}
    Let $\Q$ be the space of quadratic forms $Q$ on $V$ satisfying
    \begin{equation*}
      Q(x_1\otimes\dots\otimes x_{r+2s})=\prod_{j=1}^{r+2s}\scal{x_j}{A_jx_j},
    \end{equation*}
    where $A_i$ is a positive definite symmetric operator on $\RR^n$
    for $1\le i\le r$,
    and $A_{r+i}$ is a positive definite hermitian operator on $\CC^n$
    for $1\le i\le s$,
    and $A_{r+s+i}=\overline{A_{r+i}}$ for $1\le i\le s$,
    and
    \begin{equation*}
      \prod_{j=1}^{r+2s}\det A_j = 1.
    \end{equation*}
    Finally, let $\G$ be the linear group acting on $V$ and whose elements $g$ satisfy
    \begin{equation*}
      g(x_1\otimes\dots\otimes x_{r+2s})
       = g_1x_1\otimes\dots\otimes g_{r+2s}x_{r+2s},
    \end{equation*}
    for some $g_i\in\GL(n,\RR)$ for $1\le i\le r$,
    and some $g_{r+i}\in\GL(n,\CC)$ and $g_{r+s+i}=\overline{g_{r+i}}$ for $1\le i\le r$,
    such that
    \begin{equation*}
      \prod_{j=1}^{r+2s}\det(g_j)=1.
    \end{equation*}
    The elements $H$ of its Lie algebra $\g$ are of the form
    \begin{multline*}
      H(x_1\otimes\dots\otimes x_m) = 
      (H_1x_1\otimes x_2\otimes\dots\otimes x_m) \\
      + (x_1\otimes H_2x_2\otimes x_3\otimes\dots\otimes x_m)
      +\dots
      + (x_1\otimes\dots\otimes x_{r+2s-1}\otimes H_{r+2s}x_{r+2s}),
    \end{multline*}
    where $H_i\in\glin(n,\RR)$ for $1\le i\le r$,
    and $H_{r+i}\in\glin(n,\CC)$ and $H_{r+s+i}=\overline{H_{r+i}}$ for $1\le i\le r$,
    and
    \begin{equation*}
      \sum_{j=1}^{r+2s}\tr(H_j)=0.
    \end{equation*}
    Then $\quartet{V}{\Q}{Z}{\G}$ is a Voronoi space.
    The maximal values of the associated Hermite function
    has been in particular studied by R.~Baeza and M.~I. Icaza
    in \cite{Icaza} and \cite{BaezaIcaza}.
    
\end{enumerate}

\section{Structure of the group of Voronoi space}\label{sectGroupVoronoiSpace}

The theoretical results of our paper rely
strongly on the structure of the group of Voronoi space,
which is a selfdual real Lie group with some property of connectedness.
The main features we will extract from this structure are
the convexity property of Lemma~\ref{lemEstimationNeighborhood},
the reducibility criterion of Proposition~\ref{propWeaklyPerfectOnly},
the structure allowing a natural definition of design (Definition~\ref{defDesign}),
and the connection expressed in Theorem~\ref{thmVenkovCriterion}
between designs and notions related to the extremality problem.

In this section, we give details concerning the important aspects of the structure
of the group of Voronoi space,
relying on well-known structure results of the theory of real Lie groups.
The statements of this section
may also be readily checked for the particular examples
of Paragraph~\ref{parExamples},
but we give here the proofs for the general case.

\subsection{Structure of selfdual linear Lie groups}

Let $E$ be a Euclidean space;
we denote by $\GP(E)$ the cone of
positive definite selfadjoint linear transformations of $E$, and 
by $\gp(E)$ the vector space of selfadjoint linear operators of $E$.
Note that $\exp:\gp(E)\to\GP(E)$ is a diffeomorphism.
We have a scalar product on $\End(E)$, the associative algebra
of linear endomorphisms
of~$V$, given by
\begin{equation*}
  \scal{X}{Y} = \tr(X^*Y).
\end{equation*}
For this scalar product, the selfadjoint linear operators on~$E$
are orthogonal to the antiselfadjoint linear operators, and we have
\begin{equation*}
  \scal{X}{[Y,Z]} = \scal{[Y^*,X]}{Z},
  \quad\forall X,Y,Z\in\End(E)
\end{equation*}

\begin{Theorem}\label{thmStructureSelfdual}
  Let $E$ be a Euclidean space,
  and let $\G\subseteq\GL(E)$ be a linear Lie group such that
  \begin{equation*}
    g\in\G\quad\Rightarrow\quad
    g^*\in\G\ \ \text{and}\ \ g^*g\in\G^0.
  \end{equation*}
  Let moreover
  \begin{itemize}
    \item $\GK=\G\cap\Orth(E)$ be the subgroup of
      orthogonal elements of $\G$;
    \item $\GP=\G\cap\GP(E)$ be the set of positive definite
      selfadjoint elements of $\G$;
    \item $\g\subseteq\glin(E)$ be the Lie algebra of $\G$;
    \item $\gk=\g\cap\so(E)$ be the subalgebra of
      antiselfadjoint elements of~$\g$, or the Lie algebra of $K$;
    \item $\gp=\g\cap\gp(E)$ be the linear subspace of
      selfadjoint elements of~$\g$, or the tangent space of $\GP$.
  \end{itemize}
  Then
  \begin{enumerate}[(i)]
    \item we have $\g=[\g,\g]\oplus\gz(\g)$, and $[\g,\g]$ is semisimple;
    \item $\GP=\exp(\gp)$;
    \item $\g=\gk\oplus\gp$ and $\G=\GK\GP$; moreover, for any $g\in\G$,
      there exist unique $k\in \GK$ and $p\in\GP$ such that $g=kp$.
    \item there exists a triangular connected Lie subgroup
      $\GD\subseteq\G$
      (i.e., a connected Lie subgroup whose elements are expressed
      by upper triangular matrices with respect to some basis of $E$)
      such that $\G=\GK\GD$.
  \end{enumerate}
\end{Theorem}
  
  If $\G$ were a connected semisimple group, Claims (iii) and (iv) would be
  respectively the
  classic ``Cartan decomposition'' and ``Iwasawa decomposition'' of $\G$,
  that we recall here:
  
\begin{Theorem}\label{thmSemisimpleGroup}
  Let $\G$ be a connected semisimple group, and let $\g$ be its Lie algebra.
  Let $\g=\gk\oplus\gp$ be a Cartan decomposition of $\g$, that is
  a direct sum of vector subspaces such that the 
  map $\theta:X+Y\mapsto X-Y$ ($X\in\gk$, $Y\in\gp$)
  is an automorphism of~$\g$, and the bilinear form
  \begin{equation*}
    b_{\theta}(X,Y)= -\tr_{\g}\bigl(\ad(\theta X)\,\ad(Y)\bigr),
    \quad X,Y\in\g,
  \end{equation*}
  is positive definite. Then:
  \begin{enumerate}[(i)]
    \item\label{clmCartanSemisimpleGroup} \textit{(Cartan decomposition.)}
      Let $\GK=\exp(\gk)$. We have a diffeomorphism
      \begin{equation*}
        \varphi:\GK\times\gp\to\G, \quad \varphi(k,H) = k\exp(H).
      \end{equation*}
    \item\label{clmIwasawaSemisimpleGroup}
      \textit{(Iwasawa decomposition.)} There exists a connected triangular
      Lie subgroup $\GD\subseteq\G$
      such that we have a diffeomorphism
      \begin{equation*}
        \mu:\GK\times\GD\to\G, \quad \mu(k,d) = kd.
      \end{equation*}
  \end{enumerate}
\end{Theorem}

See, e.g., \cite[Chap.~4, Theorem~3.2, p.145, and Theorem~4.6, p.~159]{OV}.
Our proof of Theorem~\ref{thmStructureSelfdual} relies on that result;
but since our group~$\G$ is not semisimple,
we need some work before we can apply it.

\begin{Lemma}\label{lemDecomposition}
  Let $\g\subseteq\glin(E)$ be
  a linear Lie algebra such that $X^*\in\g$ whenever $X\in\g$.
  Then we have the decomposition
  \begin{equation*}
    \g = [\g,\g] \oplus \gz(\g),
  \end{equation*}
  the algebra $[\g,\g]$ is semisimple,
  the spaces $[\g,\g]$ and $\gz(\g)$ are stable under the adjunction
  $X\mapsto X^*$, and they are orthogonal one to the other.
  Moreover, the decomposition
  \begin{equation*}
    [\g,\g] = ([\g,\g]\cap\gk) \oplus ([\g,\g]\cap\gp) 
  \end{equation*}
  is a Cartan decomposition.
\end{Lemma}

\begin{proof}
  Let $X\in\g$; we have $X\in[\g,\g]^{\perp}$ if and only if
  \begin{equation*}
    0 = \scal{X}{[Y,Z]} = \scal{[Y^*,X]}{Z} \quad\forall Y,Z\in\g;
  \end{equation*}
  thus since $\g$ is stable under the adjunction, we have $[Y^*,X]=0$
  for all $Y^*\in\g$,
  that is $X\in\gz(\g)$. This shows that $\g = [\g,\g] \oplus \gz(\g)$, and
  $[\g,\g]$ and $\gz(\g)$ are orthogonal.
  
  Let now $\mathfrak{r}$ be the radical of $\g$; it is clearly stable under
  $X\mapsto X^*$, so we can apply to $\mathfrak{r}$ what we have proved
  for $\g$, so we have
  $\mathfrak{r}=[\mathfrak{r},\mathfrak{r}]\oplus\gz(\mathfrak{r})$.
  If $\mathfrak{r}'=[\mathfrak{r},\mathfrak{r}]$, we have
  \begin{equation*}
    \mathfrak{r}' = [\mathfrak{r},\mathfrak{r}]
      = [\mathfrak{r}'\oplus\gz(\mathfrak{r}), \mathfrak{r}'\oplus\gz(\mathfrak{r})]
      = [\mathfrak{r}',\mathfrak{r}'];
  \end{equation*}
  therefore, since $\mathfrak{r}$ is solvable, we have
  $[\mathfrak{r},\mathfrak{r}]=\{0\}$. Now,
  since $\mathfrak{r}$ is an ideal of $\g$, we have
  $[\mathfrak{r},[\mathfrak{r},\g]]\subseteq[\mathfrak{r},\mathfrak{r}]=\{0\}$,
  and then
  \begin{equation*}
    \scal{[R,X]}{[R,X]} = \scal{[R^*,[R,X]]}{X} =0
    \quad
    \forall R\in\mathfrak{r},\ X\in \g.
  \end{equation*}
  Therefore we have $\mathfrak{r}\subseteq\gz(\g)$,
  and $[\g,\g]$ is semisimple.

  Let $\g'=[\g,\g]$, $\gk'=[\g,\g]\cap\gk$, and $\gp'=[\g,\g]\cap\gp$,
  and let us show that $\g'=\gk'\oplus\gp'$ is a Cartan decomposition.
  The corresponding automorphism is given by $\theta(X)=-X^*$,
  and the bilinear form is
  \begin{equation*}
    b_{\theta}(X,Y)= \tr_{\g'}\bigl(\ad(X^*)\,\ad(Y)\bigr)
  \end{equation*}
  In order to show that $b_{\theta}$ is positive definite,
  it suffices to show that $\ad(X^*)=\ad(X)^*_{\g'}$,
  where $\alpha^*_{\g'}$ denotes the adjoint of $\alpha\in\glin(\g')$
  for the scalar product $\scal{X}{Y}=\tr(X^*Y)$ on $\g'$.
  We have
  \begin{equation*}
    \scal{\ad(X^*)Y}{Z} = \scal{[X^*,Y]}{Z} 
    = \scal{Y}{[X,Z]}
    = \scal{Y}{\ad(X)Z}, \quad \forall X,Y,Z\in\g'.
  \end{equation*}
  Therefore $\ad(X)^*_{\g'}=\ad(X^*)$.
\end{proof}

\begin{Lemma}\label{lemUniquenessDecomposition}
  Let $E$ be a Euclidean space, and let $k_1,k_2\in\Orth(E)$ and
  $p_1,p_2\in\GP(E)$ such that $k_1p_1=k_2p_2$.
  Then $k_1=k_2$ and $p_1=p_2$.
\end{Lemma}
\begin{proof}
  We have $p_1=k_1^*k_2p_2$, and also $p_1=p_1^*=p_2k_2^*k_1$.
  Therefore
  \begin{equation*}
    p_1^2 = (p_2k_2^*k_1)(k_1^*k_2p_2) =p_2^2.
  \end{equation*}
  Since $p_1$ and $p_2$ are positive definite, we have $p_1=p_2$.
  The equality $k_1=k_2$ follows.
\end{proof}

\begin{proof}[Proof of Theorem~\ref{thmStructureSelfdual}]
  (i) This follows from Lemma~\ref{lemDecomposition}.
    
  (ii) and (iii)
  Let $\g'=[\g,\g]$, $\gk'=[\g,\g]\cap\gk$, and $\gp'=[\g,\g]\cap\gp$.  
  According to
  Theorem~\ref{thmSemisimpleGroup}(\ref{clmCartanSemisimpleGroup}),
  we have
  \begin{equation*}
    \G'^0 = \GK'^0\exp(\gp'),
  \end{equation*}
  where $\G'^0=\exp(\g')$, $\GK'^0=\exp(\gk')$.
  It follows from Lemma~\ref{lemDecomposition} that we have
  \begin{gather*}
    \GK^0 = \GK'^0\,\bigl(\GZ(\G)^0\cap\GK\bigr), \qquad
    \exp(\gp) = \exp(\gp')\,\bigl(\GZ(\G)^0\cap\GP\bigr), \\
    \G^0 = \GK^0\exp(\gp).
  \end{gather*}
  Now, let $g\in\G$; as we have $g^*g\in\G^0$,
  there exists $\tilde{k}\in\GK^0$ and $H\in\gp$ such that
  $g^*g=\tilde{k}\exp(H)$.
  Since $g^*g$ is symmetric positive definite,
  by Lemma~\ref{lemUniquenessDecomposition},
  we have in fact $g^*g=\exp(H)$. Then we have $g=kp$ with
  \begin{equation*}
    k=g\exp(-H/2)\in\GK \quad\text{and}\quad p=\exp(H/2)\in\exp(\gp).
  \end{equation*}
  In particular, if $p\in\GP$, we have a decomposition $p=kp'$
  with $p'\in\exp(\gp)$. By Lemma~\ref{lemUniquenessDecomposition},
  we have $p=p'$; therefore, the equality $\GP=\exp(\gp)$ holds.
  The same Lemma~\ref{lemUniquenessDecomposition}
  shows the uniqueness of the decomposition $g=kp$.
  
  (iv) According to
  Theorem~\ref{thmSemisimpleGroup}(\ref{clmIwasawaSemisimpleGroup}),
  there exists a connected triangular Lie subgroup $\GD'\subseteq\G'$
  such that $\G'^0=\GK'^0\GD$. The group
  \begin{equation*}
    \GD = \GD'\,\bigl(\GZ(\G)^0\cap\GP\bigr)
  \end{equation*}
  is connected and triangular, and we have $\G^0=\GK^0\GD$.
  Now, we have
  \begin{equation*}\tag*{\qedhere}
    \GK\GD = \GK\GK^0\GD = \GK\G^0 = \GK\GK^0\GP = \GK\GP = \G.
  \end{equation*} 
\end{proof}

\subsection{Structure of the group of Voronoi space}

In order to apply Theorem~\ref{thmStructureSelfdual} to the group
of a Voronoi space, we have to check that the assumption
of the theorem is satisfied.

\begin{Lemma}\label{lemConnection}
  Let $\quartet{V}{\Q}{Z}{\G}$ be a Voronoi space, and let $Q\in\Q$.
  Then, for every $g\in\G$, we have $g^*_Qg\in\G^0$.
\end{Lemma}

\begin{proof}
  Let $g\in\G$, and let $Q'=Q\circ g$. As $\Q$ is connected, there exists
  $h\in\G^0$ such that $Q'=Q\circ h$. Then we have
  \begin{equation*}
    Q'(x,y) = Q(x,g^*_Qgy) = Q(x,h^*_Qhy), \quad\forall x,y\in V.
  \end{equation*}
  Therefore, we have $g^*_Qg=h^*_Qh\in\G^0$.
\end{proof}

As a first application of Theorem~\ref{thmStructureSelfdual},
we have the following description of the set of quadratic forms
of a Voronoi space.

\begin{Proposition}\label{propAdjoinization}
  Let $\quartet{V}{\Q}{Z}{\G}$ a Voronoi space, and let $Q\in\Q$.
  For $H\in\gp_Q$, we define $Q_H$ by
  \begin{equation*}
    Q_H(x) = Q(x,\exp(H)x).
  \end{equation*}
  Then we have
  \begin{equation*}
    \Q = \{Q_H \mid H\in\gp_Q \}.
  \end{equation*}
  In particular, if, for $Q\in\Q$ and $\epsilon>0$, we set
  \begin{equation*}
    V_{Q,\epsilon} = \bigl\{ Q_H \bigm| H\in\gp_Q\ \text{and}\
    \norm{H}_{Q}<\epsilon \},
  \end{equation*}
  then the family $\{V_{Q,\epsilon}\}_{\epsilon>0}$
  is a basis of neighborhoods of $Q$ in~$\Q$.
\end{Proposition}

\begin{proof}
  Let $Q'\in\Q$, and let $g\in\G$ such that $Q'=Q\circ g$.
  According to Theorem~\ref{thmStructureSelfdual}, we have
  $g_Q^*g\in\GP_Q=\exp(\gp_Q)$,
  therefore there exists an $H\in\gp_Q$ such that $g_Q^*g=\exp(H)$,
  and we have $Q'=Q_H$.
  Conversely, for $H\in\gp_Q$, we have $Q_H=Q\circ g$ with $g=\exp(H/2)$,
  so $Q_H\in\Q$. The rest of the lemma is straightforward.
\end{proof}

\section{Characterization of extremality \`a la Voronoi}\label{sectCharacterizationVoronoi}

The notions of perfection and eutaxy play
a key role in theories involving problems of extremality
for a Hermite-type function,
because they are the main ingredients of
the theory of Voronoi and its generalizations.

In this section,
we introduce the notion of perfection and eutaxy
in a Voronoi space,
and we prove a characterization
of extremality and strict extremality
in the spirit of the one of Voronoi
(Theorem~\ref{thmVoronoi}).
Finally, we prove a statement
(Proposition~\ref{propWeaklyPerfectOnly})
which provides an efficient tool to check
that extremality implies strict extremality
in many Voronoi spaces.

\subsection{Neighborhoods in Voronoi space}

Let $\quartet{V}{\Q}{Z}{\G}$ be a Voronoi space.
Firstly, we need estimates of the value taken by forms
in a neighborhood of a given quadratic form $Q$ in $\Q$.
Recall (Proposition~\ref{propAdjoinization}) that such quadratic forms
are of the form $Q_H(x)=Q(x,\exp(H)x)$,
with $H\in\gp_Q$ and $\norm{H}_Q$ small. Recall that we have defined
\begin{equation*}
  \norm{H}_Q^2 = \sup_{x\in V\smz} \frac{Q(Hx)}{Q(x)},
\end{equation*}
The triangular inequality implies $\abs{Q(x,Hx)} \le \norm{H}\times Q(x)$.

The following lemma states that the functions
$t\mapsto Q_{tH}(x)$ are strictly convex near to $0$,
except when they are constant.
The convexity is a key property in all characterizations of
extremality by perfection and eutaxy. 

\begin{Lemma}\label{lemEstimationNeighborhood}
  Let $Q\in\Q$ and  $0<t\le1$, and $H\in\gp_Q$ such that $\norm{H}_Q\le1$.
  For every $x\in V$, we have
  \begin{equation*}
    Q(x) + t\,Q(x,Hx) \le Q_{tH}(x) \le Q(x) + t\,Q(x,Hx) + t^2Q(Hx),
  \end{equation*}
  with equality if and only if $Hx=0$.
\end{Lemma}
\begin{proof}
  We have
  \begin{align*}
    Q_{tH}(x)
      &= Q\bigl(x,\exp(tH)x\bigr)
      =   \sum_{k\ge0} \frac{t^k}{k!} Q(x,H^kx) \\
      &= Q(x) + t\,Q(x,Hx) + \frac{t^{2\!}}{2}Q(Hx)
         + \sum_{k\ge3}\frac{t^k}{k!} Q(Hx,H^{k-1}x).
  \end{align*}
  Let $S$ be the sum at the end of that expression.
  We have $\bigl|Q(Hx,H^{k-1}x)\bigr| \le \norm{H}_Q^{k-2}\* Q(Hx)$, and therefore
  \begin{equation*}
    \mathopen|S\mathclose|
    = \biggl|\sum_{k\ge3}\frac{t^k}{k!} Q(Hx,H^{k-1}x)\biggr|
    \le t^2\sum_{k\ge3}\frac{\bigl(t\norm{H}_Q\bigr)^{k-2}}{k!} Q(Hx)
  \end{equation*}
  Now, using $t\le1$ and $\norm{H}_Q\le1$, we get
  \begin{equation*}
    \mathopen|S\mathclose|
    \le t^2 \sum_{k\ge3}\frac{1}{k!} Q(Hx)
    = t^2\bigl(e-5/2\bigr)\,Q(Hx)
    \le \frac12 t^2Q(Hx),
  \end{equation*}
  with equality if and only if $Q(Hx)=0$. From this, we deduce
  \begin{equation*}
    Q(x) + t\,Q(x,Hx) \le  Q\bigl(x,\exp(tH)x\bigr) \le Q(x) + t\,Q(x,Hx) + t^2Q(Hx),
  \end{equation*}
  with equality if and only if $Hx=0$.
\end{proof}

The next lemma is standard in all Voronoi theories.
It states that, when studying
the value of the Hermite function on quadratic forms
in a neighborhood of~$Q\in\Q$,
it suffices to consider the values taken by those quadratic forms
on the minimal layer of~$Q$.

\begin{Lemma}\label{lemMinimalVectorsInNeighborhood}
  Let $\quartet{V}{\Q}{Z}{\G}$ be a Voronoi space, and let $Q\in\Q$.
  Then there exists a neighborhood $U\subseteq\Q$ of $Q$
  such that $Q'_{\min}\subseteq Q_{\min}$ for every $Q'\in U$.
\end{Lemma}
\begin{proof}
  Let $\gamma=\gamma(Q)=\min_{x\in Z}Q(x)$
  and let $\delta=\min_{x\in Z\setminus Q_{\min}}Q(x)-\gamma >0$.
  Let $H\in\gp_Q$ with $\norm{H}_Q\le1$;
  by Lemma~\ref{lemEstimationNeighborhood}, we have, for any $x\in Q_{\min}$,
  \begin{align*}
    Q_H(x) &\le  Q(x) + Q(x,Hx) + Q(Hx)
      \le Q(x) + \norm{H}_Q Q(x) + \norm{H}_Q^2 Q(x) \\
      &\le \gamma \bigl(1+\norm{H}_Q + \norm{H}_Q^2\bigr),
  \end{align*}
  and for any $z\in Z\setminus Q_{\min}$
  \begin{equation*}
    Q_H(z) \ge  Q(z) + Q(z,Hz)
      \ge Q(z) - \norm{H}_Q Q(z) 
      \ge (\gamma+\delta)\bigl(1-\norm{H}_Q\bigr)
  \end{equation*}
  Let $\epsilon>0$ such that $\gamma (1+\eta + \eta^2) < (\gamma+\delta)(1-\eta)$
  for every $\eta$ with $0\le\eta<\epsilon$. Then $Q_H(x)< Q_H(z)$ if $\norm{H}_Q<\epsilon$.
  So we have the claim of the lemma with $U=\{Q_H\mid H\in\gp_Q, \norm{H}_Q<\epsilon\}$.
\end{proof}

\subsection{Perfection and eutaxy}

We introduce now the two fundamental notions, and variants of them,
which enters in characterizations of extremality.

\begin{Definition}\label{defPerfectionAndEutaxy}
  Let $\quartet{V}{\Q}{Z}{\G}$ be a Voronoi space.
  Let $Q\in\Q$ be a quadratic form, and $X$ a nonempty finite subset
  of $V\setminus\{0\}$.
\begin{enumerate}[(a)]
  \item $(Q,X)$ is \emph{eutactic} when, for every $ H\in\g$,
    \begin{equation*}
      \exists x\in X,\ Q(x,Hx)>0
      \quad\Longrightarrow\quad
      \exists y\in X,\ Q(y,Hy)<0.
    \end{equation*}
  \item $(Q,X)$ is \emph{strongly eutactic} when, for every $ H\in\g$,
    \begin{equation*}
      \sum_{x\in X} \frac{Q(x,Hx)}{Q(x)}=0.
    \end{equation*}
  \item $(Q,X)$ is \emph{weakly perfect} when, for every $H\in\g$,
    \begin{equation*}
      \exists c\in\RR,\ \forall x\in X,\ Q(x,Hx)=c\,Q(x)
      \quad\Longrightarrow\quad
      \forall x\in X,\ Hx+H^*_Qx = 2cx.
    \end{equation*}
  \item $(Q,X)$ is \emph{perfect} when, for every $H\in\g$,
    \begin{equation*}
      \exists c\in\RR,\ \forall x\in X,\ Q(x,Hx)=c\,Q(x)
      \quad\Longrightarrow\quad
      H+H^*_Q = 2c\,\id_V.
    \end{equation*}
\end{enumerate}
Moreover, $Q$ is \emph{eutactic}, \emph{perfect}, etc.,
when $(Q,Q_{\min})$ is eutactic, perfect, etc.
\end{Definition}

In the definition above, we can replace $H\in\g$ by $H\in\gp_Q$.
This follows directly from the fact that any $H\in\g$ can be decomposed in $\g$
as the sum of a \pf{Q}-selfadjoint element and a \pf{Q}-antiselfadjoint one,
and that the conditions of the definition are trivially true for \pf{Q}-antiselfadjoint
linear transformations.
Namely, we have:
\begin{enumerate}[(a)]
  \item $(Q,X)$ is eutactic iff, for every $ H\in\gp_Q$,
    \begin{equation*}
      \exists x\in X,\ Q(x,Hx)>0
      \quad\Longrightarrow\quad
      \exists y\in X,\ Q(y,Hy)<0.
    \end{equation*}
  \item $(Q,X)$ is strongly eutactic iff, for every $ H\in\gp_Q$,
    \begin{equation*}
      \sum_{x\in X} \frac{Q(x,Hx)}{Q(x)}=0.
    \end{equation*}
  \item $(Q,X)$ is weakly perfect iff, for every $H\in\gp_Q$,
    \begin{equation*}
      \exists c\in\RR,\ \forall x\in X,\ Q(x,Hx)=c\,Q(x)
      \quad\Longrightarrow\quad
      \forall x\in X,\ Hx = cx.
    \end{equation*}
  \item $(Q,X)$ is perfect iff, for every $H\in\gp_Q$,
    \begin{equation*}
      \exists c\in\RR,\ \forall x\in X,\ Q(x,Hx)=c\,Q(x)
      \quad\Longrightarrow\quad
      H = c\,\id_V.
    \end{equation*}
\end{enumerate}

We mention also a geometric interpretation of these notions,
although we will not use it.
Let us fix $Q\in\Q$ and consider, for $x\in V$, the linear form
$\epsilon_x\in (\gp_Q\oplus\RR\,\id_V)^*$ defined by
$\epsilon_x(H)=Q(x,Hx)/\rbreak Q(x)$.
Let $\tau$ be the linear form on $\g\oplus\RR\,\id_V$ such that
$\tau(H)=0$ for any $H\in\g$ and $\tau(\id_V)=1$;
for most interesting examples of Voronoi spaces, the form $\tau$
is the normalized trace on $V$.
Then
\begin{enumerate}[(a)]
  \item $(Q,X)$ is eutactic iff
    the form $\tau|_{\gp_Q\oplus\RR\,\id_V}$
    is in the geometric interior of the convex hull
    of $\{\epsilon_x\}_{x\in X}\subseteq (\gp_Q\oplus\RR\,\id_V)^*$.
  \item $(Q,X)$ is strongly eutactic iff
    \begin{equation*}
      \frac{1}{\card{X}}\sum_{x\in X} \epsilon_x=\tau|_{\gp_Q\oplus\RR\,\id_V}.
    \end{equation*}
  \item $(Q,X)$ is weakly perfect iff the family
    $\{\tilde{\epsilon}_x\}_{x\in X}$
    generates $\bigl((\gp_Q\oplus\RR\,\id_V)/\mathcal{A}\bigr)^*$, where
    \begin{equation*}
      \mathcal{A} = \{H\in\gp_Q\mid Hx = 0,\ \forall x\in X\},
    \end{equation*}
    and $\tilde{\epsilon}_x(H+\mathcal{A}) = \epsilon_x(H)$.
  \item $(Q,X)$ is perfect iff the family $\{\epsilon_x\}_{x\in X}$
    generates $(\gp_Q\oplus\RR\,\id_V)^*$.
\end{enumerate}
The geometric reformulation of Claim (a)
can be deduced from the Hahn-Banach theorem;
other claims follow from standard finite-dimensional
linear space theory.

\subsection{Characterization of extremality}

The following result, that we call the
\emph{Voronoi characterization of extremality},
characterizes extreme and strictly extreme forms
in terms of perfection and eutaxy.
The characterization of strict extremality is 
known for many particular cases: it
has been proved by Voronoi in the case of
classic Voronoi space \cite{Voronoi},
by Berg\'e and Martinet in the cases
of dual-extremality in classic Voronoi space
\cite[Th\'eor\`eme~3.14]{DualExtreme},
and family of quadratic forms in classic Voronoi space
\cite[Th\'eor\`eme~2.10]{ExtremeAutomorphism},
\cite[Th\'eor\`eme~4.5]{LatticeFamily},
by Coulangeon in the case of exterior powers of classic Voronoi space
\cite[Th\'eor\`eme~3.2.3]{CoulangeonkExtreme},
and Humbert forms \cite[Proposition~3]{CoulangeonVenkov}.
Moreover, in many of the above cases, extremality is shown
to be equivalent to strict extremality; but this question will
be considered in the next paragraph.
However, we do not know 
any appearance in the literature of a
characterization of extremality, not necessarily strict,
similar to the one we give here%
---although this absence is justified in the many cases where
extremality implies strict extremality---,
but a weaker result appears in \cite[Th\'eor\`eme~4.5]{LatticeFamily}.

\begin{Theorem}[(Voronoi characterization)]\label{thmVoronoi}
  Let $\quartet{V}{\Q}{Z}{\G}$  be a Voronoi space.
  \begin{enumerate}[(i)]
    \item A quadratic form $Q\in\Q$ is strictly extreme if and only if
    it is perfect and eutactic.
    \item A quadratic form $Q\in\Q$ is extreme if and only if
      there exists a nonempty subset of $Q_{\min}$
      that is weakly perfect and eutactic.
  \end{enumerate}
\end{Theorem}

Before giving the proof, we reformulate the conditions of perfection and eutaxy:

\begin{Lemma}\label{lemPerfectionAndEutaxy}
  Let $\quartet{V}{\Q}{Z}{\G}$ be a Voronoi space, and
  let $X$ be a nonempty finite subset of $V$.
  \begin{enumerate}[(i)]
    \item $(Q,X)$ is perfect and eutactic if and only if,
      for each $H\in\gp_Q$, one of the two following conditions holds:
      \begin{enumerate}[(a)]
        \item $H=0$;
        \item there exists $x\in X$ such that $Q(x,Hx)<0$.
      \end{enumerate}
    \item $(Q,X)$ is weakly perfect and eutactic if and only if, for each $H\in\gp_Q$,
      one of the two following conditions holds:
      \begin{enumerate}[(a)]
        \item $Hx=0$ for every $x\in X$;
        \item there exists $x\in X$ such that $Q(x,Hx)<0$.
      \end{enumerate}
  \end{enumerate}
\end{Lemma}
\begin{proof}
  Let $H\in\gp_Q$.
  The eutaxy of $(Q,X)$ implies that
  Condition~(b) holds unless
  we have $Q(x,Hx)=0$ for every $x\in X$.
  In the last case, the perfection, respectively the weak perfection,
  of $(Q,X)$
  implies Condition~(a).
\end{proof}

\begin{proof}[Proof of the Theorem]
  First, recall that a basis of neighborhoods of~$Q\in\Q$ is given by
  $\bigl\{Q_H \bigm| H\in\nobreak\gp_Q$ and $\norm{H}_Q<\epsilon\bigr\}$,
  with $\epsilon>0$ (Proposition~\ref{propAdjoinization}).

  (i, $\Leftarrow$) Suppose that $Q$ is perfect and eutactic.
  Consider the following function on the sphere $\Se(\gp_Q)$ of
  elements of $\gp_Q$ whose \pf{Q}-norm is equal to~$1$:
  \begin{equation*}
    m(H) = \min_{x\in Q_{\min}} Q(x,Hx), \quad H\in\Se(\gp_Q).
  \end{equation*}
  The perfection and eutaxy of $Q$ imply that $m(H)<0$
  for each $H\in\Se(\gp_Q)$ (Lemma~\ref{lemPerfectionAndEutaxy}).
  Then, the continuity of~$m$ and the compactness of $\Se(\gp_Q)$ imply that
  there exists a~$\delta>0$ such that $m(H)<-\delta$ for every $H\in\Se(\gp_Q)$.
  We infer that
  \begin{equation*}
   \min_{x\in Q_{\min}} Q(x,Hx) < -\delta\norm{H}_Q, \quad \forall H\in\pp_Q\smz.
  \end{equation*}
  Now, let $H\in\gp_Q\smz$ with $\norm{H}_Q<1$, and let $x\in Q_{\min}$
  such that $Q(x,Hx)<-\delta\norm{H}_Q$.
  By Lemma~\ref{lemEstimationNeighborhood},
  \begin{equation*}
    Q_{H}(x) \le Q(x) + Q(x,Hx) + Q(Hx)
    < Q(x) - \delta\norm{H}_Q + \norm{H}_Q^2 Q(x).
  \end{equation*}
  Therefore, when $\norm{H}_Q<\min\bigl(1,\delta/Q(x)\bigr)$,
  we have $\gamma(Q_H)\le Q_{H}(x)<Q(x)=\gamma(Q)$.
  So, $Q$ is strictly extreme.

  (ii, $\Leftarrow$)
  The pattern of the proof is similar as the one for (i, $\Leftarrow$) above.
  Suppose that there exists $X\subseteq Q_{\min}$ such that $(Q,X)$
  is weakly perfect and eutactic. Let
  $\mathcal{A} = \{ H\in\nobreak\gp_Q \mid Hx=0,\ \rbreak\forall x\in\nobreak X\}$.
  We have a norm $\norm{\cdot}_Q$ on $\gp_Q/\mathcal{A}$ defined by
  \begin{equation*}
    \norm{H+\mathcal{A}}_Q = \min_{J\in\mathcal{A}} \norm{H+J}_Q,
    \quad H\in\gp_Q.
  \end{equation*}
  As $(H+J)x=Hx$ for every $J\in\mathcal{A}$ and every $x\in X$, we have
  \begin{equation*}
    \bigl|Q(x,Hx)\bigr|\le\norm{H+\mathcal{A}}_Q Q(x)
    \quad\text{and}\quad
    Q(Hx)\le\norm{H+\mathcal{A}}_Q^2 Q(x),
    \quad\forall x\in X.
  \end{equation*}
  We define now a function~$m$ on the sphere
  $\Se(\gp_Q/\mathcal{A})$ consisting of classes of norm~$1$ by
  \begin{equation*}
    m(H+\mathcal{A}) = \min_{x\in X} Q(x,Hx),
    \quad H\in\gp_Q,\ \norm{H+\mathcal{A}}_Q=1.
  \end{equation*}
  Now we use the weak perfection and the eutaxy of $(Q,X)$ :
  Lemma~\ref{lemPerfectionAndEutaxy} implies that $m(H+\mathcal{A})$
  is negative on $\Se(\gp_Q/\mathcal{A})$;
  therefore, by compactness, there exists a~$\delta>0$ such that
  $m\bigl(H+\mathcal{A}\bigr)<-\delta$
  for every $H+\mathcal{A}\in\Se(\gp_Q/\mathcal{A})$. We infer that
  \begin{equation*}
   \min_{x\in Q_{\min}} Q(x,Hx) \le -\delta\norm{H+\mathcal{A}}_Q, \quad \forall H\in\pp_Q.
  \end{equation*}
  Now, let $H\in\gp_Q$ with $\norm{H}_Q<1$,
  and let $x\in X$ such that $Q(x,Hx) \le -\delta\*\norm{H+\mathcal{A}}_Q$.
  By Lemma~\ref{lemEstimationNeighborhood}, we have
  \begin{equation*}
    Q_{H}(x) \le Q(x) + Q(x,Hx) + Q(Hx) 
    \le Q(x) - \delta\norm{H+\mathcal{A}}_Q + \norm{H+\mathcal{A}}_Q^2 Q(x).
  \end{equation*}
  Therefore, when $\norm{H+\mathcal{A}}_Q<\min\bigl(1,\delta/Q(x)\bigr)$,
  we have $\gamma(Q_H)\le Q_{H}(x)\le Q(x)=\gamma(Q)$. So, $Q$ is extreme.

  (i, $\Rightarrow$) Suppose that $Q$ is strictly extreme.
  Let $H\in\gp_Q\smz$ with $\norm{H}_Q\le1$.
  By Lemma~\ref{lemEstimationNeighborhood}, we have, for each $x\in Q_{\min}$,
  \begin{equation*}
    Q_{tH}(x) \ge Q(x) + t\,Q(x,Hx), \quad
    \forall t,\ 0<t<1.
  \end{equation*}
  Since $Q$ is strictly extreme, and $(Q_{tH})_{\min}\subseteq Q_{\min}$ 
  for $t$ sufficiently small (Lemma~\ref{lemMinimalVectorsInNeighborhood}),
  we have $Q_{tH}(x)<Q(x)$ for some $x\in Q_{\min}$ and some $t>0$.
  This implies $Q(x,Hx)<0$ for some $x\in Q_{\min}$, which means,
  by Lemma~\ref{lemPerfectionAndEutaxy}, that $(Q,Q_{\min})$ is both perfect and eutactic.

  (ii, $\Rightarrow$) We proceed by induction on the cardinal of $Q_{\min}$.
  Suppose that $Q$ is extreme.
  Let $0<\delta\le1$ such that,
  for every $H\in\gp_Q$ with $\norm{H}_Q\le\delta$, we have
  $(Q_{H})_{\min}\subseteq Q_{\min}$
  (by Lemma~\ref{lemMinimalVectorsInNeighborhood})
  and $\gamma(Q_H)\le\gamma(Q)$ (by the assumption of extremality).
  Suppose that $(Q,Q_{\min})$ is \emph{not} weakly perfect and eutactic (otherwise we are done).
  Then, by Lemma~\ref{lemPerfectionAndEutaxy}, there exists an $H\in\gp_Q$
  and an $x_0\in Q_{\min}$,
  such that $Q(x,Hx)\ge0$ for every $x\in Q_{\min}$ and $Hx_0\ne0$.
  Moreover, we can choose $H$ such that $\norm{H}_Q<\delta$.
  By Lemma~\ref{lemEstimationNeighborhood}, we have,
  \begin{equation*}
    Q_{H}(x) \ge Q(x) + Q(x,Hx),\ \forall x\in Q_{\min},
    \text{with a strict inequality when $Hx\ne0$.}
  \end{equation*}
  Then the extremality of~$Q$ and the choice of $\delta$
  imply that $\gamma(Q_H)=\gamma(Q)$, with
  \begin{equation*}
    (Q_H)_{\min} = \{x\in Q_{\min} \mid Hx=0\},
  \end{equation*}
  and $\bigl|(Q_H)_{\min}\bigr|<\bigl|Q_{\min}\bigr|$ (since $Hx_0\ne0$),
  and $Q_H$ is itself extreme.
  By induction, there exists a subset $X$ of $(Q_H)_{\min}$
  such that $(Q_H,X)$ is weakly perfect and eutactic.
  It remains to show that $(Q,X)$ is also weakly perfect and eutactic.

  On the one hand, we have, for every $x\in X$ and every $J\in\g$,
  using $Hx=0$,
  \begin{equation*}
    Q_H(x,Jx) = Q(\exp(H)x,Jx) = Q(x,Jx).
  \end{equation*}
  On the other hand, we have for every $x\in X$ and every $y\in V$,
  \begin{multline*}
    Q(J^*_Qx,y) = Q(x,Jy) = Q_H(\exp(-H)x,Jy) \\
    = Q_H(x,Jy) = Q_H(J^*_{Q_H}x,y) = Q(\exp(H)J^*_{Q_H}x,y),
  \end{multline*}
  so $J^*_{Q_H}x=\exp(-H)J^*_Qx$ for every $x\in X$.

  Let us now show that $(Q,X)$ is weakly perfect and eutactic
  using Lemma~\ref{lemPerfectionAndEutaxy}.
  Let $J\in\gp_Q$.
  If there exists $x\in X$ such that $Q_H(x,Jx)<0$, then,
  since $Q_H(x,Jx)=  Q(x,Jx)$ for every $x\in X$,
  there exists $x\in X$ such that $Q(x,Jx)<0$.
  Otherwise, since $(Q_H,X)$ weakly perfect and eutactic,
  by Lemma~\ref{lemPerfectionAndEutaxy}, we have $Jx+J^*_{Q_H}x=0$ for every $x\in X$.
  But since,$J^*_{Q_H}x=\exp(-H)J^*_Qx$, we have $(1+\exp (-H))\,Jx=0$.
  Yet $(1+\exp(-H))$ is an isomorphism, because it is positive \pf{Q}-definite.
  Therefore $Jx = 0$ for every $x\in X$.
\end{proof}

\subsection{More about perfection}

It is known that, in classic Voronoi spaces and many other
interesting Voronoi spaces,
extremality implies strict extremality.
This is not true in general;
but the following proposition gives a simple property
of weakly perfect quadratic form,
which turns to be an efficient tool to
rule out the existence of extreme but not strictly extreme quadratic forms
in many Voronoi spaces.

\begin{Proposition}\label{propWeaklyPerfectOnly}
  Let $\quartet{V}{\Q}{Z}{\G}$ be a Voronoi space,
  and let $Q\in\Q$ be a quadratic form
  and $X\subseteq V$ be a nonempty finite subset.
  If $(Q,X)$ is weakly perfect but not perfect,
  then there exists a proper \pf{\G^0}-invariant subspace $U$ of $V$ such that $X\subseteq U$.
\end{Proposition}

\begin{proof}
  Let $Q\in\Q$ and $X\subseteq V$ as in the hypothesis of the theorem.
  Consider the following subspace of $\gp_Q\oplus\RR\,\id_V$:
  \begin{equation*}
    \mathcal{A} = \{ H\in\gp_Q\oplus\RR\,\id_V \mid Hx=0,\ \forall x\in X\}.
  \end{equation*}
  Since $(Q,X)$ is weakly perfect but not perfect,
  $\mathcal{A}$ is not reduced to $\{0\}$.
  We define $U$ as
  \begin{equation*}
    U = \bigcap_{H\in\mathcal{A}}\ker H,
  \end{equation*}
   which is a proper subspace of $V$ containing~$X$.
   It remains to show that it is \pf{\G^0}-invariant, or, equivalently, \pf{\g}-invariant.

  (a) Let us first show that $U$ is invariant under the action of~$\gk_Q$.
  Let $H\in\gk_Q$. For any $J\in\mathcal{A}$ and any $x\in X$, we have
  \begin{equation*}
    Q\bigl([J,H]x,x\bigr)
    = Q\bigl(JHx-HJx,x\bigr)
    \stackrel{Jx=0} = Q\bigl(JHx,x\bigr)
    \stackrel{J=J^*_Q} = Q\bigl(Hx,Jx\bigr)
    \stackrel{Jx=0} = 0.
  \end{equation*}
  Since $[J,H]\in\gp_Q$ and since $(Q,X)$ is weakly perfect, we have $[J,H]\in\mathcal{A}$.
  Then, for any $u\in U$, we have
  \begin{equation*}
    JHu = [J,H]u + HJu = 0,
  \end{equation*}
  so $Hu\in\ker J$ for any $J\in\mathcal{A}$, therefore $Hu\in U$ by definition of~$U$.

  (b) Let us now show that the elements of $\gp_Q$ preserve~$U$;
  since $\g=\gk_Q\oplus\gp_Q$, this will achieve the proof.
  Let $H\in\gp_Q$. For any $J\in\mathcal{A}$, and any $u\in U$, we have
  \begin{multline*}
    Q\bigl([J,H]u,[J,H]u\bigr)
    = Q\bigl(JHu-HJu,[J,H]u\bigr)\\
    \stackrel{Ju=0} = Q\bigl(JHu,[J,H]u\bigr)
    \stackrel{J=J^*_Q} = Q\bigl(Hu,J[J,H]u\bigr).
  \end{multline*}
  Because we have $[J,H]\in\gk_Q$ and the elements of $\gk_Q$ preserve $U$
  by the part~(a) of the proof,
  we have $[J,H]u\subseteq U$, and therefore $J[J,H]u=0$.
  We infer that $Q\bigl([J,H]u,[J,H]u\bigr)=0$, that is $[J,H]u=0$. Then
  \begin{equation*}
    JHu = [J,H]u + HJu = 0,
  \end{equation*}
  so $Hu\in\ker J$ for any $J\in\mathcal{A}$,
  therefore $Hu\in U$ by definition of~$U$.
\end{proof}

The last proposition above implies that, in a Voronoi space
$\quartet{V}{\Q}{Z}{\G}$
where an orbit of an element of $Z$ under the action of $\G^0$
generates the space $V$, 
weak perfection implies perfection,
and therefore, by Theorem~\ref{thmVoronoi},
extremality implies strict extremality.
This is the case in particular for
\begin{itemize}
  \item the classic Voronoi space of rank~$n$
    (Example~\ref{exClassical} of Paragraph~\ref{parExamples}).
    Here, we have $\G^0=\SL(n,\RR)$, which has two orbits on $V$,
    namely $\{0\}$ and $V\smz$;
  \item the $m$th exterior power
    $\quartet{V^{\wedge m}}{\Q^{\wedge m}}{Z^{\wedge m}}{\G^{\wedge m}}$
    of the classic Voronoi space
    of rank~$n$, with $m\le n/2$
    (Example~\ref{exExterior} of Paragraph~\ref{parExamples}).
    Here, the orbit of an element of $Z^{\wedge m}$ under the action
    of $\G^0$ is equal to
    \begin{equation*}
      \{ x_1\wedge\dots\wedge x_m \mid
      \text{$x_1,\dots,x_m$ are linearly independent elements of $\RR^n$} \},
    \end{equation*}
    and it spans $V^{\wedge m}$;
  \item the space $\quartet{V}{\Q}{Z}{\G}$ of Humbert forms of rank~$n$
    over a number field
    (Example~\ref{exHumbert} of Paragraph~\ref{parExamples}).
    Here, the orbit of an element of $Z$ under the action
    of $\G^0$ is equal to
    \begin{multline*}
      \{ x_1\otimes\dots\otimes x_{r+2s} \mid
      \text{$x_i\in\RR^n\smz$ for $1\le i\le r$,} \\
      \text{$x_{r+i}\in\CC^n\smz$ and $x_{r+s+i}=\overline{x_{r+i}}$
        for $1\le i\le s$} \},
    \end{multline*}
    and this set spans $V$.
  \item the ``symplectic'' Voronoi space of rank~$2m$:
    Let $\sigma=\left(\begin{smallmatrix}0 & I_m\\-I_m & 0\end{smallmatrix}\right)$,
    and consider the Voronoi space of \pf{\sigma}-isodual
    quadratic form of the classic Voronoi space of rank~$2m$,
    as constructed in Example~\ref{exOrbit}(b)
    of Paragraph~\ref{parExamples}.
    Here, we have $\G^0=\Sp(2m)^0$, which has two orbits on $V$,
    namely $\{0\}$ and $V\smz$.
\end{itemize}
  
As another example, let $\quartet{V}{\Q}{Z}{\G}$ be a classic Voronoi space,
and let $\quartet{\tilde{V}}{\tilde{\Q}}{\tilde{Z}}{\tildeG}$
be the Voronoi space constructed
in Example~\ref{exDuality} of Paragraph~\ref{parExamples}
(with $Q^\alpha=Q^{-1}$).
For $Q\in\Q$ and $c\in\RR^*$, let $\tilde{Q}_c\in\tilde{\Q}$
such that $\tilde{Q}_c((x_1,\nobreak x_2))=Q(cx_1)+Q^{-1}(c^{-1}x_2)$.
If $c$ is chosen such that $c^2\gamma(Q)<c^{-2}\gamma(Q^{-1})$, we have
$(\tilde{Q}_c)_{\min}=Q_{\min}\times\{0\}$, and $\tilde{Q}_c$ is
weakly perfect if and only if $Q$ is perfect; but $\tilde{Q}_c$
is not perfect. However $\tilde{Q}_c$ is not eutactic either, and
it is easy to show that all examples of weakly perfect but not
perfect quadratic forms in this space are of this type.
Therefore, in this space, extremality implies strict extremality.

\bigskip
Another question is the finiteness (up to equivalence) of perfect forms
as it is the case for the classic Voronoi space.
We conjecture that the set of perfect forms in a 
Voronoi space is closed and discrete,
but we have not been able to prove it, except
under a very restrictive hypothesis, namely when the
Lie algebra $\g\oplus\RR\,\id_V$ is also an associative algebra.
Since the set of Voronoi spaces given in Paragraph~\ref{parExamples} for
which that condition is satisfied
is a strict subset of the set of Voronoi spaces for which
finiteness of perfect forms is already known, we will not give the proof.

Also, an algorithm which would allow to enumerate theoretically
all perfect form, as the one of Voronoi~\cite{Voronoi}, \cite[Chap.~7]{MartinetBook},
which has been extended in various situations
\cite[Chap.~13]{MartinetBook}, \cite{BavardSymplectique}, \cite{BavardLorentzien},
would be a subject of high interest.
However, a study of such an algorithm, is beyond the scope of this paper.

\section{Alternative Voronoi spaces}\label{sectAQV}

The classic problems related to the Hermite function can be
equivalently formulated
in terms of quadratic forms or of Euclidean lattices.
While the former
formulation is more convenient for the proof of the Voronoi
characterization of extremality,
the latter one is more appropriate for expressing
the relation between spherical designs and extremality.

The same situation holds in our frame of Voronoi spaces.
Hitherto, we have worked with the ``quadratic forms'' formulation,
in which the subset $Z$ of $V$ is fixed
and the quadratic forms $Q$ varies.
But in order to define designs,
we need to work with the dual formulation,
where the quadratic form $Q$ is fixed and the corresponding subset $Z$
of $V$ varies;
we name \emph{alternative Voronoi space} this dual concept of Voronoi space.
The Voronoi characterization of extremality
is reformulated in terms of
alternative Voronoi space in Theorem~\ref{thmVoronoiBis}.

\emph{Regular alternative Voronoi spaces}
as defined in Paragraph~\ref{parRegularAQV}
are a particular type of alternative Voronoi spaces
for which designs may be defined in the most natural way.
Finally, in the last paragraph of this section (Paragraph~\ref{parExamplesBis}),
we give the dual formulation
of all the examples of Voronoi spaces given in Paragraph~\ref{parExamples},
and we add another example of alternative Voronoi space.

\subsection{Alternative Voronoi space}

\begin{Definition}\label{defAlternativeVoronoiSpace}
  \textrm{(Compare with Definition~\ref{defVoronoiSpace}.)}
  An \emph{alternative Voronoi space} is the data of
  \begin{enumerate}[(a)]
    \item a finite-dimensional real vector space $V$,
    \item a positive definite quadratic form $Q$ of $V$,
    \item a closed (for the topology of simple convergence) set $\Z$
    of closed discrete subsets $Z$ of $V$ not containing~$0$ and spanning~$V$,
    \item a closed linear Lie group $\G<\GL(V)$,
  \end{enumerate}
  satisfying the following properties:
  \begin{enumerate}[(i)]
    \item For $Z\in\Z$ and $c>0$, we have $cZ\in\Z$ only if $c=1$;
    \item for all $g\in\G$, we have $g^*_Q\in\G$ and $g^*_Qg\in\G^0$;
    \item for all $g\in\G$ and $Z\in\Z$, we have $gZ\in\Z$,
      and $\G$ is transitive on $\Z$.
  \end{enumerate}
\end{Definition}
The supplementary requirement $g^*_Qg\in\G^0$ in condition~(ii)
is to be compared to Lemma~\ref{lemConnection}; it replaces the
condition of connectedness of $\Q$ in the definition of
Voronoi space.

\begin{Definition}\label{defGroupBis}
  \textrm{(Compare with Definition~\ref{defGroup}.)}
  Let $\quartet{V}{Q}{\Z}{\G}$ be an alternative Voronoi space.
  We define the following objects:
  \begin{itemize}
    \item $\GK=\{g\in\G \mid g^*_Q=g^{-1}\}$, the group of \pf{Q}-selfadjoint
      elements of~$\G$,
      which is the \emph{compact group} of $\quartet{V}{Q}{\Z}{\G}$;
    \item $\g$, the Lie algebra of~$\G$;
    \item $\gk = \{H\in\g\mid H+H^*_Q=0\}$, the subalgebra of \pf{Q}-antiselfadjoint
      elements of~$\g$, or the Lie algebra of~$\GK$;
    \item $\gp = \{H\in\g\mid H-H^*_Q=0\}$, the set of \pf{Q}-selfadjoint
      elements of~$\g_Q$;
    \item $\G^0$ and $\GK^0$, the connected component of the identity
      of $\G$ and $\GK$ respectively.
  \end{itemize}
  Note that we have $\g=\gk\oplus\gp$.
\end{Definition}

\begin{Definition}\label{defLayerAQV}
  \textrm{(Compare with Definition~\ref{defLayerQV}.)}
  Let $\quartet{V}{Q}{\Z}{\G}$ be an alternative Voronoi space.
  \begin{enumerate}[(i)]
    \item The \emph{Hermite function} associated to $\quartet{V}{Q}{\Z}{\G}$
    is the function $\gamma$ defined on $\GK\backslash\Z$ by
      \begin{equation*}
        \gamma(\GK Z) = \min_{x\in Z} Q(x), \quad Z\in\Z.
      \end{equation*}
    \item For $r>0$, the \emph{layer} of square radius~$r$ of $Z\in\Z$ is the set
      \begin{equation*}
        Z_{r} = \{x\in Z \mid Q(x)=r \}.
      \end{equation*}
    \item The set of \emph{minimal vectors} of a set $Z\in\Z$ is
      the layer of minimal radius, that is
      \begin{equation*}
        Z_{\min} = Z_{\gamma(Z)} =\{x\in Z \mid Q(x)=\gamma(Q) \}.
      \end{equation*}
    \item A set $Z\in\Z$ is \emph{extreme},
      respectively \emph{strictly extreme},
      when its class $\GK Z$
      is a local maximum, respectively a strict local maximum, of~$\gamma$.
  \end{enumerate}
\end{Definition}

\begin{Definition}\label{defCompactBis}
  \textrm{(Compare with Definition~\ref{defCompact}.)}
  The alternative Voronoi space $\quartet{V}{Q}{\Z}{\G}$ is \emph{compact}
  when \begin{math}\gamma: \GK\backslash\Z
  \to \mathopen]0,\infty\mathclose[\end{math}
  is bounded and proper.
\end{Definition}

\subsection{Equivalence between Voronoi space and alternative Voronoi space}

We construct standard transformations from Voronoi spaces
to alternative Voronoi spaces and conversely as follows:

  Let $\quartet{V}{\Q}{Z}{\G}$ be a Voronoi space, and let $Q\in\Q$.
  We construct an alternative Voronoi space $\quartet{V}{Q}{\Z}{\G}$ by setting
  \begin{equation*}
    \Z = \{gZ \mid g\in \G\}.
  \end{equation*}
  If $\gamma$ is the Hermite function of $\quartet{V}{\Q}{Z}{\G}$ and
  $\tilde{\gamma}$ the Hermite function of $\quartet{V}{Q}{\Z}{\G}$,
  we have $\tilde{\gamma}(gZ) = \gamma(Q\circ g)$.

Conversely, let $\quartet{V}{Q}{\Z}{\G}$ be an alternative Voronoi space,
and let $Z\in\Z$.
We construct a Voronoi space $\quartet{V}{\Q}{Z}{\G}$ by setting
\begin{equation*}
  \Q = \{Q\circ g \mid g\in \G\}.
\end{equation*}
The property $g^*_Qg\in\G^0$ for $g\in\G$ implies that
$\Q$ is connected.
If $\gamma$ is the Hermite function of $\quartet{V}{Q}{\Z}{\G}$ and
$\tilde{\gamma}$ the Hermite function of $\quartet{V}{\Q}{Z}{\G}$,
we have $\tilde{\gamma}(Q\circ g) = \gamma(gZ)$.

The Voronoi characterization of extremality (Theorem~\ref{thmVoronoi})
can be formulated
in the language of alternative Voronoi space.
First we have to complete the definitions of eutaxy and perfection
(Definition~\ref{defPerfectionAndEutaxy}):
\begin{Definition}\label{defPerfectionAndEutaxyBis}
  Let $\quartet{V}{Q}{\Z}{\G}$ be an alternative Voronoi space, and let $Z\in\Z$.
  Then $Z$ is \emph{eutactic}, \emph{perfect}, etc., when $(Q,Z_{\min})$ is
  eutactic, perfect, etc.
\end{Definition}

\begin{Theorem}[(Voronoi characterization)]\label{thmVoronoiBis}
(Compare with Theorem~\ref{thmVoronoi}.) Let $\quartet{V}{Q}{\Z}{\G}$  be an
alternative Voronoi space.
  \begin{enumerate}[(i)]
    \item A set $Z\in\Z$ is strictly extreme if and only if
      it is perfect and eutactic.
    \item A set $Z\in\Z$ is extreme if and only if there exists
      a nonempty subset of $Z_{\min}$ that is weakly perfect and eutactic.
  \end{enumerate}
\end{Theorem}

\begin{proof}
  Let $Z\in\Z$, and let $\quartet{V}{\Q}{Z}{\G}$ be the Voronoi space
  associated to $\quartet{V}{Q}{\Z}{\G}$ and $Z$.
  Then the extremality, perfection, eutaxy, etc., of $Z$
  in the alternative Voronoi space $\quartet{V}{Q}{\Z}{\G}$
  is equivalent the extremality, perfection, eutaxy, etc., of $Q$
  in the Voronoi space $\quartet{V}{\Q}{Z}{\G}$.
  Therefore, the theorem is equivalent to the
  corresponding theorem for Voronoi spaces (Theorem~\ref{thmVoronoi}).
\end{proof}

Finally, we give a description of neighborhoods in alternative Voronoi space:
\begin{Proposition}\label{propAdjoinizationBis} (Compare with Proposition~\ref{propAdjoinization}.)
  Let $\quartet{V}{Q}{\Z}{\G}$ an alternative Voronoi space, and let $Z\in\Z$.
  For $H\in\gp$, we define $Z_H$ by
  \begin{equation*}
    Z_H = \exp(H/2)Z.
  \end{equation*}
  Then each $Z'\in\Z$ can be written as
  $Z=kZ_H$ for some $H\in\gp$ and some $k\in\GK$.
  In particular, if, for $Z\in\Z$ and $\epsilon>0$, we set
  \begin{equation*}
    V_{Z,\epsilon} = \bigl\{ \GK Z_H\in\GK\backslash\Z \bigm| H\in\gp\ \text{and}\
    \norm{H}_{Q}<\epsilon \},
  \end{equation*}
  then the family $\{V_{Z,\epsilon}\}_{\epsilon>0}$
  is a basis of neighborhoods of $\GK\backslash\Z$.
\end{Proposition}

\begin{proof}
  Let $Z'\in\Z$, and let $g\in\G$ such that $Z'=gZ$.
  According to Theorem~\ref{thmStructureSelfdual}, we have
  $g_Q^*g\in\GP=\exp(\gp)$,
  therefore there exists an $H\in\gp$ such that $g_Q^*g=\exp(H)$.
  Let $k=g\exp(-H/2)$; we have $k\in\GK$ and $g=k\exp(H)$,
  and we have $Z'=k Z_H$.
  The rest of the lemma is straightforward.
\end{proof}

\subsection{Regular alternative Voronoi space}\label{parRegularAQV}

There is not, in general, a natural notion of design
for an arbitrary alternative Voronoi space; but we need to restrict
to the following class
of alternative Voronoi spaces.

\begin{Definition}\label{defRegularAQV}
  Let $\quartet{V}{Q}{\Z}{\G}$ be an alternative Voronoi space, and let
  \begin{equation*}
    \Omega = \overline{\bigcup_{Z\in\Z}[Z]}\in\PP(V),
  \end{equation*}
  where $[\ ]$ denotes the projection $V\smz\to\PP(V)$ and
  $\overline{E}$ denotes
  the adherence of~$E$.
  The space $\quartet{V}{\Q}{Z}{\G}$ is \emph{regular} when
  $\Omega$ is an orbit of $\G$ acting on $\PP(V)$.
\end{Definition}

The following fact is crucial for the definition of designs.

\begin{Proposition}\label{propRegularAQV}
  Let $\quartet{V}{Q}{\Z}{\G}$ be a regular alternative Voronoi space.
  Then its compact group $\GK$
  is transitive on $\Omega$.
\end{Proposition}

In order to show that statement,
we need the following fact on triangular linear Lie groups:

\begin{Lemma}\label{lemActionTriangular}
  Let $\GD$ be a connected triangular subgroup of $\GL(V)$
  and let $\Omega$ be a \pf{\GD}-invariant closed subset
  of the projective space $\PP(V)$.
  Then $\GD$ has a fixed point in~$\Omega$.
\end{Lemma}
For a proof, see~\cite[Chap.~4, Prop.~4.6, p.~161]{OV}.

\begin{proof}[Proof of Proposition~\ref{propRegularAQV}]
  According to Theorem~\ref{thmStructureSelfdual},
  there exists a connected triangular subgroup $\GD$ of $\G$
  such that $\G=\GK\GD$. By Lemma~\ref{lemActionTriangular},
  there exists $\xi\in\Omega$ such that $\GD\xi=\xi$.
  As the alternative Voronoi space is regular, its group $\G$ is transitive on $\Omega$,
  and we have
  \begin{equation*}
    \Omega = \G\xi = \GK\GD\xi = \GK\xi,
  \end{equation*}
  therefore $\GK$ is transitive on $\Omega$.
\end{proof}

\subsection{Examples}\label{parExamplesBis}

We give here the alternative Voronoi spaces corresponding
to the Voronoi spaces
given in Paragraph~\ref{parExamples},
and we add another example of alternative Voronoi space.

\begin{enumerate}[1.]
  \item \textit{Classic spaces.} Let $V=\RR^n$ with $n\ge2$,
    let $Q$ be the canonical quadratic form on $\RR^n$, let
    $\Z$ be the set of Euclidean lattices in $\RR^n$ of determinant~$1$
    without their zero vector,
    and let $\G=\SL^{\pm}(n,\RR)$.
    Then $\quartet{V}{Q}{\Z}{\G}$
    is a regular alternative Voronoi space with $\Omega=\PP(\RR^n)$.
    
  \item\label{exDualityAQV} \textit{Duality.}
    Let $\quartet{V}{Q}{\Z}{\G}$ be an alternative Voronoi space.
    Let $Z\mapsto Z^\alpha$ be an involution of~$\Z$,
    and $g\mapsto g^\alpha$ be an involution and automorphism of~$\G$
    such that $(g^*_Q)^\alpha = (g^\alpha)^*_Q$ and
    \begin{equation*}
      (gZ)^\alpha = g^\alpha Z^\alpha,\quad
      \forall g\in G,\ \forall Z\in\Z.
    \end{equation*}
    We define an alternative Voronoi space
    $\quartet{\tilde{V}}{\tilde{Q}}{\tilde{\Z}}{\tildeG}$ as follows: 
    \begin{gather*}
      \tilde{V} = V\times V, \\
      \tilde{Q}\bigl((x_1,x_2)\bigr) = Q(x_1)+Q(x_2),\\
      \tilde{\Z} = \bigl\{(cZ\times\{0\})\cup (\{0\}\times c^{-1}Z^\alpha)
        \bigm| Z\in\Z,\ c\in\RR^*\bigr\},\\
      \begin{split}
        \tildeG = \Bigl\{ \tilde{g} \in\GL(V\times V)
          &\Bigm| \tilde{g} =
            \begin{pmatrix} cg & 0 \\ 0 & c^{-1}g^\alpha \end{pmatrix}\ \text{or}\
            \begin{pmatrix} 0 & cg \\ c^{-1}g^\alpha & 0 \end{pmatrix}\\[-2\jot]
          &\qquad\qquad\qquad\text{for some $g\in\G$ and some $c\in\RR^*$} \Bigr\}.
      \end{split}
    \end{gather*}
    We have
    \begin{equation*}
      \tilde{\Omega} = \bigl(\Omega\times\{0\}\bigr)
        \cup \bigl(\{0\}\times\Omega\bigr),
    \end{equation*}
    and the space $\quartet{\tilde{V}}{\tilde{Q}}{\tilde{\Z}}{\tildeG}$ is regular
    if $\quartet{V}{Q}{Z}{\G}$ is regular.
    The local maxima of the Hermite function $\tilde{\gamma}$ of
    $\quartet{\tilde{V}}{\tilde{Q}}{\tilde{\Z}}{\tildeG}$
    correspond to the local maxima of the function
    $\gamma'(Z)= \bigl(\gamma(Z)\*\,\gamma(Z^\alpha)\bigr)^{1/2}$
    on $\Z$.
    
    As a particular case, we take for $\quartet{V}{Q}{\Z}{\G}$ the classic alternative
    Voronoi space of rank~$n$, with $g^\alpha=g^{{*}{-1}}$ and
    $Z^\alpha=Z^*$, where $Z^*\cup\{0\}$ is the dual lattice
    of $Z\cup\{0\}$.
    
  \item\label{exOrbitAQV} \textit{Family of lattices.}
    Let $\quartet{V}{Q}{\Z}{\G}$ be an alternative Voronoi space,
    let $Z$ be an element of~$\Z$, 
    and let $\mathbf{H}$ be a connected closed subgroup of $\G$
    such that $g^*_Q\in\mathbf{H}$ when $g\in\mathbf{H}$.
    Let $\mathbf{H}Z$ be the orbit of $Z$ in $\Z$.
    Then $\quartet{V}{Q}{\mathbf{H}Z}{\mathbf{H}}$ is an alternative Voronoi space,
    which is in general not regular, even when $\quartet{V}{Q}{\Z}{\G}$ is regular.
    
    As special cases, we have:
    \begin{enumerate}[(a)]
      \item 
        Let $\mathbf{\Gamma}$ be a finite subgroup of $\GK$,        
        let $Z\in\Z$ be a \pf{\mathbf{\Gamma}}-invariant set,
        that is such that $\gamma Z=Z$ for every $\gamma\in\mathbf{\Gamma}$,
        and let $\mathbf{H}$ be the connected component of the identity
        of the centralizer of $\mathbf{\Gamma}$ in $\G$.
        Then $\mathbf{H}$ is \pf{Q}-selfadjoint and
        $\mathbf{H}Z$ is
        a connected component of
        the set of \pf{\mathbf{\Gamma}}-invariant elements of~$\Z$.
        (Note that, contrary to what happens with the dual formulation
        in terms of Voronoi space,
        $\mathbf{H}Z$ is not necessary the entire set of
        \pf{\mathbf{\Gamma}}-invariant elements of~$\Z$.)
        So, $\quartet{V}{Q}{\mathbf{H}Z}{\mathbf{H}}$
        is an alternative Voronoi space.

      \item Let $Z\mapsto Z^\alpha$ and $g\mapsto g^\alpha$
        as in Example~\ref{exDualityAQV}; let $\sigma\in\G$ such that
        $(\sigma^*)^{-1}=\sigma$.
        An element $Z\in\Z$ such that $Z^\alpha=\sigma Z$ is called 
        \emph{\pf{\sigma}-isodual}.

        Let $Z$ be a \pf{\sigma}-isodual element of $\Z$.
        Let $\mathbf{H}$ be the connected component of the identity
        of the elements of $g\in\G$
        satisfying $\sigma g=g^\alpha\sigma$.
        Then $\mathbf{H}$ is \pf{Q}-selfadjoint,
        and $\mathbf{H}Z$ is a connected component of
        the set of \pf{\sigma}-isodual elements of~$\Z$.
        So, $\quartet{V}{Q}{\mathbf{H}Z}{\mathbf{H}}$ is an alternative Voronoi space.
    \end{enumerate}   

  \item\label{exExteriorAQV} \textit{Exterior powers or Grassmannian spaces.}
    Let $\quartet{V}{Q}{\Z}{\G}$ be an alternative Voronoi space
    and let $1\le m\le\dim(V)/2$.
    We define a space $\quartet{V^{\wedge m}}{Q^{\wedge m}}{\Z^{\wedge m}}{\G^{\wedge m}}$
    as follows: $V^{\wedge m}$ is the $m$th exterior power of $V$,
    and
    \begin{gather*}
      Q^{\wedge m}(x_1\wedge\dots\wedge x_m)
        = \det\bigl(Q(x_i,x_j)_{i,j=1}^{m}\bigr), \\
      \Z^{\wedge m}=\{Z^{\wedge m}\mid Z\in\Z\},\qquad
      \G^{\wedge m} = \{\tilde{g} \mid g\in\G\},
    \end{gather*}
    where
    \begin{gather*}
        Z^{\wedge m} = \{x_1\wedge\dots\wedge x_m
        \mid \text{$x_1,\dots,x_m$ are linearly independent elements of $Z$}\},\\
      \tilde{g}(x_1\wedge\dots\wedge x_m) = gx_1\wedge\dots\wedge gx_m.
    \end{gather*}
    Note that the fact that $\quartet{V}{Q}{\Z}{\G}$ is regular
    does not imply that the space
    $\quartet{V^{\wedge m}}{Q^{\wedge m}}{\Z^{\wedge m}}{\G^{\wedge m}}$
    is regular.
    If $\quartet{V}{Q}{\Z}{\G}$ is the classic Voronoi space of rank~$n$,
    then $\quartet{V^{\wedge m}}{Q^{\wedge m}}{\Z^{\wedge m}}{\G^{\wedge m}}$ is regular, with
    \begin{equation*}
      \Omega^{\wedge m} = \bigl\{[x_1\wedge\dots\wedge x_m]\bigm|
        \text{$x_1, \dots,x_m$ are linearly independent elements of $\RR^n$}\bigr\},
    \end{equation*}
    which is the space of nonoriented Grassmannians of dimension~$m$
    in~$\RR^n$.

  \item \textit{Humbert forms.}
    Let $K$ be a number field of signature $(r,s)$,
    of integral ring $\mathcal{O}_K$,
    let $\sigma_1,\dots,\sigma_r$ the real embeddings of $K$,
    and $\sigma_{r+1},\dots,\sigma_{r+2s}$ its complex embeddings,
    with $\sigma_{r+s+i}=\overline{\sigma_{r+i}}$,
    and let $n$ be a positive integer.
    In what follows, `$\otimes$'' denotes always the tensor product as
    \textit{real} vector space.
    Let $V$ be the linear subspace of $(\RR^n)^{\otimes r}\otimes(\CC^n)^{\otimes 2s}$
    generated by $x_1\otimes\dots\otimes x_{r+2s}$, where
    \begin{gather*}
      x_i\in\RR^n \quad \text{for $1\le i\le r$},\\
      x_{r+i}\in\CC^n \ \text{and}\ x_{r+s+i}=\overline{x_{r+i}}\quad \text{for $1\le i\le s$}.
    \end{gather*}
    Let $Q$ be the quadratic form on $V$ defined by
    \begin{equation*}
      Q(x_1\otimes\dots\otimes x_{r+2s})=\prod_{j=1}^{r+2s}\scal{x_j}{x_j}.
    \end{equation*}   
    Let $\G$ be the linear group acting on $V$ and whose elements $g$ satisfy
    \begin{equation*}
      g(x_1\otimes\dots\otimes x_{r+2s})
       = g_1x_1\otimes\dots\otimes g_{r+2s}x_{r+2s},
    \end{equation*}
    for some $g_i\in\GL(n,\RR)$ for $1\le i\le r$,
    and some $g_{r+i}\in\GL(n,\CC)$ and $g_{r+s+i}=\overline{g_{r+i}}$ for $1\le i\le r$,
    such that
    \begin{equation*}
      \prod_{j=1}^{r+2s}\det(g_j)=1.
    \end{equation*}
    Let $\Z = \{gZ\mid g\in\G\}$, where
    \begin{equation*}
      Z=\{\sigma_1(x)\otimes\dots\otimes \sigma_{r+s}(x)
      \mid x\in\mathcal{O}_n\}.
    \end{equation*}
    We obtain a regular alternative Voronoi space $\quartet{V}{\Q}{Z}{\G}$ with
    \begin{multline*}
      \Omega=\bigl\{ [x_1\otimes\dots\otimes x_{r+2s}] \bigm|
      \text{$x_i\in\RR^n\smz$ for $1\le i\le r$,} \\
      \text{$x_{r+i}\in\CC^n\smz$ and $x_{r+s+i}=\overline{x_{r+i}}$
        for $1\le i\le s$} \bigr\}.
    \end{multline*}
    
  \item\label{exIsotropy} \textit{Isotropy in isodual lattices.}
    We add an example of alternative Voronoi space, which does not seem
    to have a natural formulation
    in terms of Voronoi space.
    
    Let $V=\RR^n$. An \emph{isodual} lattice
    is the data of a Euclidean lattice $\Lambda\subseteq\RR^n$
    and an orthogonal transformation $\sigma\in\Orth(n)$ such that
    $\Lambda^*=\sigma(\Lambda)$.
    The lattice is \emph{unimodular} when $\sigma=\pm\id$,
    \emph{orthogonal} when $\sigma^2=\id$ and $\sigma\ne\id$,
    and \emph{symplectic} when $\sigma^2=-\id$.
    For each isodual lattice $(\Lambda,\sigma)$, we set
    \begin{equation*}
      Z_{(\Lambda,\sigma)} = \{ z\in\Lambda \mid \scal{z}{\sigma z} = 0\} \smz.
    \end{equation*}
    Let $\sigma\in\Orth(n)$;
    let $\Z$ be a connected component of the set of $Z_{(\Lambda,\sigma)}$,
    where $\Lambda$ is a lattice such that $(\Lambda,\sigma)$ is isodual;
    let $Q$ be the canonical quadratic form on $\RR^n$,
    and let
    \begin{equation*}
      \G = \{g\in\GL(n,\RR) \mid \sigma g={g^*}^{-1}\sigma \}.
    \end{equation*}
    Then $\quartet{\RR^n}{Q}{\Z}{\G}$ is an alternative Voronoi space
    of Lie algebra
    \begin{equation*}
      \g = \{H\in\glin(n,\RR) \mid \sigma H + H\sigma =0 \}.
    \end{equation*}
    The case where $\sigma^2=\id$ and $\sigma$ is of signature $({n-1},\rbreak1)$
    has been studied by Ch.~Bavard in~\cite{BavardLorentzien}.
    
\end{enumerate}

\section{Designs}\label{sectDesign}

The term of ``spherical design'' was introduced in a paper
of Delsarte, Goethals and Seidel
\cite{DelsarteGoethalsSeidel} as
an equivalent on the Euclidean sphere of the notion
of design in combinatorics, although it is also a
particular case of the notion of cubature formula on sphere,
as studied by Sobolev \cite{Sobolev}, \cite{SobolevBook}.

Boris Venkov has shown that
a lattice whose minimal vectors form a spherical \pf{4}-design
is extreme \cite{VenkovMartinet1}, criterion which has been extended
in few other situations.

In this section, we introduce an appropriate
notion of designs in regular alternative Voronoi spaces,
for which the criterion of Venkov holds (Corollary~\ref{corPerfectionCriterion}).
The frame in which we define design
is a particular case of the notion of
\emph{polynomial space} as defined and studied in \cite{PacheThesis}.
Other noteworthy, although more restrictive
notions of polynomial space are found, e.g., in
\cite[Chap.~14]{Godsil} and \cite{Gaeta}.
In the present paper, we do not develop the theory more than
needed for our purpose.

In the last paragraph, we give a useful criterion
(Corollary~\ref{corInvarianceDesign})
which allows to check that a set is a design
by considering the invariants of his automorphism group.

\subsection{Definitions}
We fix $\quartet{V}{Q}{\Z}{\G}$ a regular alternative 
Voronoi space (Definition~\ref{defRegularAQV}).
Recall that $\GK$ denotes the compact subgroup of \pf{Q}-orthogonal elements of~$\G$,
and $\gp$ denotes the subspace of \pf{Q}-selfadjoint elements of the Lie algebra
$\g$ of $\G$.
As the group $\GK$ is compact and transitive on $\Omega$ (Proposition~\ref{propRegularAQV}),
there is a unique \pf{\GK}-invariant probability measure~$\mu$ on $\Omega$.

Consider the following linear subspaces of $\End(V)$, the associative algebra
of linear endomorphisms of~$V$:
\begin{gather*}
  \pp_0 = \RR\,\id_V, \qquad
  \pp_1 = \RR\,\id_V + \gp,\\
  \pp_k = \Span\{H_1H_2\dotsm H_k+H_kH_{k-1}\dotsm H_1\mid H_i\in\gp\oplus\RR\,\id_V\},\ k\ge1.
\end{gather*}
The elements of $\pp_k$ are \pf{Q}-selfadjoint, and
we have, for $i,j\ge0$,
\begin{equation*}
  H\in\pp_i,\ J\in\pp_j\ \Rightarrow\ HJ+JH\in\pp_{i+j}.
\end{equation*}
In other words, $\pp_{i+j}$ is the subspace of \pf{Q}-selfadjoint elements
of the space generated by multiplications of elements of $\pp_i$
with elements of $\pp_j$.

Now we define the following finite-dimensional subspaces
of the space $\mathcal{C}(\Omega,\RR)$
of real continuous functions on $\Omega$:
Let $\{2k_1,\dots,2k_d\}$ be a finite set with multiplicities
whose elements are even positive integers.
We set:
\begin{equation*}
  \Fc_\emptyset = \RR\,\id_\Omega
  \qquad
  \Fc_{\{2k_1,\dots,2k_d\}} = \Span\Bigl\{ x\mapsto \prod_{i=1}^{d} \frac{Q(x,H_ix)}{Q(x)}
    \Bigm| H_i\in\pp_{k_i} \Bigr\}.
\end{equation*}
and, for $\tau$ an even nonnegative integer,
\begin{equation*}
  \Fc_\tau = \sum_{\substack{\{2k_1,\dots, 2k_d\}\\2k_1+\dots+2k_d\le\tau}}\Fc_{\{2k_1,\dots,2k_d\}}.
\end{equation*}
We have
\begin{gather*}
  \Fc_{\{2k_1,\dots 2k_d\}}\subseteq\Fc_{\{2k'_1,\dots 2k'_{d'}\}}
    \quad\text{if $d\le d'$ and $2k_i\le 2k'_i$ for $i\le d$,} \\
  \Fc_{\{2k_1,\dots 2k_d\}}\Fc_{\{2k_{d+1},\dots 2k_{d+d'}\}} = \Fc_{\{2k_1,\dots 2k_{d+d'}\}},
\end{gather*}
For $f\in\mathcal{C}(\Omega,\RR)$, we define his \emph{average} on $\Omega$ by
\begin{equation*}
  \ave{f} = \int_\Omega f(\xi)\,d\mu(\xi),
\end{equation*}
where $\mu$ is the unique \pf{\GK}-invariant probability measure on $\Omega$.

\begin{Definition}\label{defDesign}
  Let $A$ be either a finite set with multiplicities of even positive integers,
  or an even positive integer.
  A \emph{weighted design} or \emph{cubature formula of strength $A$} on $\Omega$,
  or, in short a \emph{weighted \pf{A}-design}, is the data of
  a finite subset $X$ of $\Omega$ and
  a function $W:X\to]0,1]$, $\xi\mapsto W_\xi$, such that
  $\sum_{\xi\in X} W_\xi=1$ and
  \begin{equation*}
    \sum_{\xi\in X} W_\xi f(\xi) = \ave{f}, \quad \forall f\in\Fc_A.
  \end{equation*}
  A \emph{design of strength $A$} on $\Omega$, or an \emph{\pf{A}-design}, is
  a finite subset $X$ of $\Omega$ such that
  $(X,W)$ is a weighted \pf{A}-design for $W$ the constant function
  of value $\card{X}^{-1}$.
\end{Definition}

In particular, \pf{\{2\}}-designs and \pf{2}-designs are the same things,
while \pf{4}-designs are sets which are simultaneously \pf{\{4\}}-  and \pf{\{2,2\}}-designs.

The notion of \pf{\tau}-design on $\Omega=\PP(\RR^n)$
for classic alternative Voronoi spaces
coincides with the usual notion of
antipodal spherical \pf{\tau}-designs on the
Euclidean sphere of dimension $n-\nobreak1$,
as defined in \cite{DelsarteGoethalsSeidel}
(see also \cite{GoethalsSeidel}, \cite{VenkovMartinet1}, etc.)
When our alternative Voronoi space is an exterior power of the
classic alternative Voronoi space,
$\Omega$ is identified with a nonoriented Grassmannian manifold,
and the notion of \pf{\tau}-design coincides with the one defined and studied
in \cite{DesignGrassmannian} and \cite{DesignGrassmannian2}.
(Caution: the sets with multiplicities of even positive integers we use to
define designs are not directly related to the partitions of even positive
integers used to define designs on Grassmannian manifolds
in the cited papers.)

When the Lie algebra $\g\oplus\RR\,\id_V$ is also an associative algebra
(as it is the case for classic alternative Voronoi spaces),
then \pf{\tau}-designs are equivalent to
$\{2,2,\dots,2\}$-designs, where the multiplicity
of $2$ is equal to $\tau/2$.

\subsection{Criterion for eutaxy and perfection}

Although we have defined perfection and eutaxy for subsets of $V$,
the definition is adaptable to subsets of $\PP(V)$, as follows:
\begin{Definition}\label{defPerfectionAndEutaxyProjective}
  Let $Q$ be a quadratic form on $V$, and let $X\subseteq\PP(V)$
  be a nonempty finite subset.
\begin{enumerate}[(a)]
  \item $(Q,X)$ is \emph{eutactic} when, for every $ H\in\gp$,
    \begin{equation*}
      \exists [x]\in X,\ \frac{Q(x,Hx)}{Q(x)}>0
      \quad\Longrightarrow\quad
      \exists [y]\in X,\ \frac{Q(y,Hy)}{Q(y)}<0.
    \end{equation*}
  \item $(Q,X)$ is \emph{strongly eutactic} when, for every $ H\in\gp$,
    \begin{equation*}
      \sum_{[x]\in X}\frac{Q(x,Hx)}{Q(x)}=0
    \end{equation*}
  \item $(Q,X)$ is \emph{weakly perfect} when, for every $H\in\gp$,
    \begin{equation*}
      \exists c\in\RR,\ \forall [x]\in X,\ \frac{Q(x,Hx)}{Q(x)}=c
      \quad\Longrightarrow\quad
      \forall [x]\in X,\ Hx = cx.
    \end{equation*}
  \item $(Q,X)$ is \emph{perfect} when, for every $H\in\gp$,
    \begin{equation*}
      \exists c\in\RR,\ \forall [x]\in X,\ \frac{Q(x,Hx)}{Q(x)}=c
      \quad\Longrightarrow\quad
      H = c\,\id_V.
    \end{equation*}
\end{enumerate}
\end{Definition}

The main goal of this section is to prove the following result,
which has been proved by Venkov in the case of classic
alternative Voronoi spaces \cite[Proposition~6.2 and Th\'eor\`eme~6.4]{VenkovMartinet1},
and by Bachoc, Coulangeon and Nebe in the case of Grassmannians
\cite[Theorem~6.2]{DesignGrassmannian}.

\begin{Theorem}\label{thmVenkovCriterion}\
\begin{enumerate}[(i)]
  \item If $(X,W)$ is a weighted design of strength $\{2\}$ then $(Q,X)$ is eutactic.
  \item If $X$ is a design of strength $\{2\}$ then $(Q,X)$ is strongly eutactic.
  \item If $(X,W)$ is a weighted design of strength $\{2,2\}$ then $(Q,X)$ is perfect and eutactic.
\end{enumerate}
\end{Theorem}

For this, we first need some lemmas:

\begin{Lemma}\label{lemNoConstantFunction}
  Let $H\in\gp$ and let $f([x])=Q(x,Hx)/Q(x)$.
  If $f$ is constant on $\Omega$, then $H=0$.
\end{Lemma}
\begin{proof}
  First, we reformulate the assumption of the lemma as follows:
  There exists a constant $c\in\RR^*$ such that
  \begin{equation*}
    Q\bigl(x,(H-c)x\bigr) = 0 \quad \forall [x]\in\Omega.
  \end{equation*}
  Now, for $[y]\in\Omega$ and $t\in\RR$, we have
  \begin{equation*}
    \bigl[\exp(t(H-c))\,y\bigr] = \bigl[e^{-tc}\exp(tH)\,y\bigr]
    =\bigl[\exp(tH)\,y\bigr]\in\Omega,
  \end{equation*}
  because $\exp(tH)$ is an element of $\G$ and $\Omega$ is \pf{\G}-invariant.
  Therefore we have
  $Q\bigl(\exp(t(H-c))\,y,\rbreak
  (H-\nobreak c)\*\exp(t(H-c))\*\,y\bigr)=0$,
  or, using $H^*_Q=H$,
  \begin{equation*}
    Q\bigl(\exp(2t(H-c))y,(H-c)y\bigr)=0, \quad
    \forall y\in\Omega,\ t\in\RR.
  \end{equation*}
  Differentiating this expression in $t$ and setting $t=0$, we obtain
  $2Q\bigl((H-c)y,\rbreak (H-c)y\bigr)=0$,
  therefore $Hy=cy$ for every $y\in\Omega$.
  As $\Omega$ spans $V$, we have $H=c\,\id_V$,
  and since $\RR\,\id_V\cap\gp=\{0\}$, we have $c=0$.
\end{proof}

\begin{Lemma}\label{lemAverageQuadraticFunction}
  Let $H\in\gp$ and let $f([x])=Q(x,Hx)/Q(x)$.
  Then $\ave{f}=0$.
\end{Lemma}
\begin{proof}
  Using the transitivity of $\GK$ on $\Omega$ (Proposition~\ref{propRegularAQV}), we get, for any $[y]\in\Omega$,
  \begin{align*}
    \ave{f}
      &= \int_\Omega \frac{Q(x,Hx)}{Q(x)}d\mu([x])
      =  \int_{\GK} \frac{Q(gy,Hgy)}{Q(gy)}dg \\
      &= \int_{\GK}\frac{Q(y,g^{-1}Hgy)}{Q(y)}dg
      =  \frac{Q\bigl(y,\int_{\GK}g^{-1}Hg\,dg\,y\bigr)}{Q(y)} \\
      &=  \frac{Q(y,H^{\GK}y)}{Q(y)},
  \end{align*}
  where $H^{\GK}=\int_{\GK}g^{-1}Hg\,dg$ is an element of $\gp$.
  Applying Lemma~\ref{lemNoConstantFunction} to $H^{\GK}$, we get $H^{\GK}=0$;
  it follows that $\ave{f}=0$.
\end{proof}

\begin{proof}[Proof of Theorem~\ref{thmVenkovCriterion}]
  (i) Suppose that $(X,W)$ is a weighted design of strength $\{2\}$.
  Let $H\in\gp$, and let $f([x])=Q(x,Hx)/Q(x)$. As~$\ave{f}=0$
  by Lemma~\ref{lemAverageQuadraticFunction}, we get
  \begin{equation*}
    \sum_{x\in X} W_x\frac{Q(x,Hx)}{Q(x)}=0.
  \end{equation*}
  This implies that, if there exists an $x\in X$ such that $Q(x,Hx)/Q(x)>0$,
  then there exists an $y\in X$ such that $Q(y,Hy)/Q(y)<0$,
  that is $(Q,X)$ is eutactic.

  (ii) The argument is the same as in the proof of Claim~(i).

  (iii) Suppose that $(X,W)$ is a weighted design of strength $\{2,2\}$.
  As it is in particular a weighted design of strength $\{2\}$,
  $(Q,X)$ is eutactic by Claim~(i).
  It remains to show that $(Q,X)$ is perfect.

  Let $H\in\gp$, and suppose that there exists a $c\in\RR$
  such that $Q(x,Hx)/Q(x)=c$ for every $x\in X$.
  Consider the nonnegative function $f\in\Fc_{\{2,2\}}$ defined by
  \begin{equation*}
    f([x]) = \biggl(\frac{Q(x,Hx)}{Q(x)}-c\biggr)^2.
  \end{equation*}
  We have
  \begin{equation*}
    \int_{\Omega} f(x)\,dx = \sum_{x\in X} W_xf(x) = 0,
  \end{equation*}
  and therefore $f(x)=0$ for every $x\in\Omega$, that is
  \begin{equation*}
    \frac{Q(x,Hx)}{Q(x)} = c,\quad \forall x\in\Omega.
  \end{equation*}
  By Lemma~\ref{lemNoConstantFunction}, $H=0$ (and $c=0$).
  Thus, $(Q,X)$ is perfect.
\end{proof}
As a direct consequence of Theorem~\ref{thmVenkovCriterion},
using Theorem~\ref{thmVoronoiBis}, we get:

\begin{Corollary}[(Venkov criterion)]\label{corPerfectionCriterion}
  Let $\quartet{V}{Q}{\Z}{\G}$ be an alternative Vo\-ro\-noi space,
  and let $Z\in\Z$.
  If the minimum layer $Z_{\min}$ of $Z$ is a \pf{\{2,2\}}-design,
  then $Z$ is strictly extreme.
\end{Corollary}

In order to apply this criterion, we need some method for
checking easily that a given finite set of~$V$ provide a design.
In the next paragraph,
we give a criterion exploiting the symmetry of the set to be checked.

\subsection{Designs invariant under finite group}\label{parDesignInvariant}

Let $\quartet{V}{Q}{\Z}{\G}$ be a regular alternative Voronoi space,
and $\Omega$, $\Fc_A$, etc., as above.
For $A$ a set with multiplicities of even positive integers, or
an even positive integer, and for $\mathbf{F}$ a closed subgroup
of $\GK$ we note $\Fc_A^\mathbf{F}$ the subspace of \pf{A}-invariant
elements of~$\Fc_A$, that is
\begin{equation*}
  \Fc_A^\mathbf{F} = \{f\in\Fc_A \mid f\circ g = f,\ \forall g\in\mathbf{F} \}.
\end{equation*}
The technique of the following proposition
has already been used by Sobolev
in order to compute cubature formulas on spheres
\cite{Sobolev}; see also \cite[Chap.~2, \S~2, Theorem~2.3]{SobolevBook}.

\begin{Proposition}\label{propInvarianceSet}
  Let $\quartet{V}{Q}{\Z}{\G}$ be a regular alternative Voronoi space.
  Let $\mathbf{F}$ be a finite subgroup of $\GK$,
  let $X$ be a finite subset of $\Omega$,
  and $W:X\to\RR$ a function with $W(\xi)=W_\xi>0$ for all $\xi\in X$ and
  $\sum_{\xi\in X}W_\xi=1$.
  Suppose that $(X,W)$ is invariant under the action of $\mathbf{F}$, that is
  $g\xi\in X$
  and $W_{g\xi}=W_\xi$ for $g\in \mathbf{F}$ and $\xi\in X$.
  
  Then $(X,W)$ is a weighted \pf{A}-design if and only if
  \begin{equation*}
    \sum_{\xi\in X} W_\xi f(\xi) = \ave{f},
    \quad \forall f\in\Fc_A^\mathbf{F}.
  \end{equation*}
\end{Proposition}

\begin{proof}
  Let $f\in\Fc_A$.
  The invariance of $(X,W)$ implies that
  \begin{equation*}
    \sum_{\xi\in X}W_\xi f(\xi) = \sum_{\xi\in X}W_\xi f(g\xi), \quad \forall g\in \mathbf{F}.
  \end{equation*}
  Taking the sum on all elements $g\in \mathbf{F}$, we obtain
  \begin{equation*}
    \sum_{\xi\in X}W_\xi f(\xi) = \sum_{\xi\in X}W_\xi f^\mathbf{F}(\xi), \quad
    \text{where}\  f^{\mathbf{F}} = \sum_{g\in \mathbf{F}} f\circ g.
  \end{equation*}
  In the same way, the invariance of the measure $\mu$
  under the action of $\mathbf{F}\subseteq\GK$
  implies that
  \begin{equation*}
    \ave{f} = \ave{f^\mathbf{F}}.
  \end{equation*}
  Therefore, we have
  \begin{equation*}
    \sum_{\xi\in X} W_\xi f(\xi) = \ave{f},
    \quad\text{iff}\quad
    \sum_{\xi\in X} W_\xi f^\mathbf{F}(\xi) = \ave{f^\mathbf{F}}.   
  \end{equation*}
  The function $f^\mathbf{F}$ is an element of $\Fc_A$,
  which is invariant under
  the action of~$\mathbf{F}$.
  It follows that, in order to have
  the equality $\sum_{\xi\in X} W_\xi\* f(\xi) = \ave{f}$
  for every $f\in\Fc_A$, it suffices to check it
  for every $f\in\Fc_A^\mathbf{F}$.
\end{proof}

As a particular case of the last proposition,
we have:

\begin{Corollary}\label{corInvarianceDesign}
  Under the assumptions of
  Proposition~\ref{propInvarianceSet},
  if the only \pf{\mathbf{F}}-invariant functions in $\Fc_A$
  are the constant functions,
  then $(X,W)$ is a weighted \pf{A}-design.
\end{Corollary}

In particular, using
the Venkov criterion of extremality (Corollary~\ref{corPerfectionCriterion}),
we have:

\begin{Corollary}\label{corInvarianceCriterion}
  Let $\quartet{V}{Q}{\Z}{\G}$ be an alternative Voronoi space,
  and let $Z\in\Z$. Let $\mathbf{F}=\{g\in \GK\mid gZ=Z\}$.
  If the only \pf{\mathbf{F}}-invariant functions in $\Fc_{\{2,2\}}$
  are the constant functions, then $Z$ is strictly extreme.
\end{Corollary}

For example, let $\Quartet{(\RR^n)^{\wedge m}}{Q}{\Z_{n,m}}{\SL^{\pm}(n,\RR)}$ be the $m$th exterior power
of the classic alternative Voronoi space of rank~$n$
(Example~\ref{exExteriorAQV} of Paragraph~\ref{parExamplesBis}).
Recall that the elements of $\Z_{n,m}$ are of the form
\begin{equation*}
  Z(\Lambda,m) = \bigl\{x_1\wedge\dots\wedge x_m \bigm|
  \text{$x_1, \dots,x_m$ are linearly independent elements of $\Lambda$}\bigr\},
\end{equation*}
where $\Lambda$ is a Euclidean lattice in $\RR^n$ of determinant~$1$;
and the space $\Omega_{n,m}=\overline{[\bigcup_{Z\in\Z_{n,m}}Z]}$
is identified to the space of \pf{m}-dimensional subspaces of $\RR^n$.
We denote by $\Fc_\tau(\Omega_{n,m})=\Fc_\tau$ the corresponding spaces
of function on $\Omega_{n,m}$.
Consider the following list of Euclidean lattices in $\RR^n$
(the subscript indicates the dimension of the ambient space),
which is taken from \cite{Bachoc} (see also \cite{VenkovMartinet1}):
\begin{multline*}
  \mathbf{A}_2,\ \mathbf{D}_4=\mathit{BW}_4,\
  \mathbf{E}_6,\ \mathbf{E}_7,\ \mathbf{E}_8=\mathit{BW}_8,\\
  Q_{14},\ \Lambda_{16}=\mathit{BW}_{16},\ O_{16},\
  O_{22},\ \Lambda_{22},\ \Lambda_{22}[2],\ M_{22}, M_{22}[5],\\
  O_{23},\ \Lambda_{23},\ M_{23},\ M_{23}[2],\ \Lambda_{24},\
  \text{and}\ \mathit{BW}_{2^k}, k\ge2,
\end{multline*}
where $\mathbf{A}_n,$ $\mathbf{D}_n$, $\mathbf{E}_n$ denote root lattices,
$\Lambda_n$ denotes laminated lattices, and
$\mathit{BW}_n$ denotes the Barnes-Wall lattices.
It is known that the automorphism group of these lattices
has no nonconstant invariant polynomial in
$\Fc_4(\Omega_{n,m})$ for every $m\le n/2$ (see \cite{Bachoc});
therefore, all these lattices and their duals provide strictly
extreme sets for the alternative Voronoi spaces
$\Quartet{(\RR^n)^{\wedge m}}{Q}{\Z_{n,m}}{\SL^{\pm}(n,\RR)}$.

\section{Epstein zeta function}\label{sectEpstein}

Hitherto, we have considered extremality relatively to the Hermite function;
now, we turn towards extremality relatively to the Epstein zeta function
as defined below.
Delone and Ryshkov have formulated
a characterization of the so-called \emph{final \pf{\zeta}-extremality} \cite{DeloneRyshkov}
in what we call classic Voronoi spaces.
This characterization
has some resemblance with Voronoi characterization of extremality,
as it involves the notions of eutaxy and perfection.
Likewise, a criterion of \pf{\zeta}-extremality using designs
resembling to the criterion of Venkov for usual extremality
has been found by Coulangeon~\cite{CoulangeonEpstein}.

The goal of this section is to prove both results in our
frame of Voronoi space.
In this whole section, we suppose, for simplicity,
that the Voronoi spaces $\quartet{V}{\Q}{Z}{\G}$ we consider
are antipodal, that is $Z=-Z$.
All the results are also correct for non-antipodal spaces,
and even for sets $Z$ with positive weights attached to its elements
and to the corresponding terms in the Epstein \pf{\zeta}-series,
if correct weights are set in all conditions of eutaxy and designs.

\subsection{Definitions}

Let $V$ be a vector space, $Q$ a positive definite quadratic form,
and $Z$ a subset of~$V\smz$, discrete and closed in~$V$. We define the
\emph{Epstein zeta series} of $(Q,Z)$ by
\begin{equation*}
  \zeta(Q,Z,s) = \sum_{x\in Z} Q(x)^{-s}, \quad s\in\CC.
\end{equation*}
When $Q$, respectively $Z$, is implied,
we write $\zeta(Z,s)$, respectively $\zeta(Q,s)$,
instead of $\zeta(Q,Z,s)$.

\begin{Lemma}\label{lemConvergence}
  There exists a $s_0(Z)\in[-\infty,+\infty]$ (with $s_0\ge0$ when $Z$ is infinite)
  depending only on $Z$ such that
  $\zeta(Q,Z,s)$ converges absolutely for $s\in\CC$ with $\Re s>s_0$
  and diverges for $s\in\RR$ with $s>s_0$.
\end{Lemma}
\begin{proof}
  Suppose that $\zeta(Q,Z,s)$ converges for some positive definite quadratic form $Q$
  and some real number $s$,
  and let $Q'$ be another positive definite quadratic form, and $s'\in\CC$ such that $\Re s'< s$.
  There exists $\alpha>0$ such that $Q'(x)\le\alpha Q(x)$ for every $x\in V$.
  We have
  \begin{equation*}
    \sum_{x\in Z}\abs{Q'(x)^{-s'}}
    = \sum_{x\in Z}Q'(x)^{-\Re s'}
    \le \alpha^{-s}\sum_{x\in Z}Q(x)^{-s}
    =\alpha^{-s}\zeta(Q,Z,s)<\infty.
  \end{equation*}
  So, we have the lemma with $s_0(Z)$ the supremum of the values of $s$
  for which $\zeta(Q,Z,s)$ converges.
\end{proof}

\begin{Lemma}\label{lemDifferentiable}
  Let $\quartet{V}{\Q}{Z}{\G}$ be a Voronoi space, let $s_0=s_0(Z)$ as in the previous lemma.
  Then, for every $s>s_0$, the function $Q\mapsto\zeta(Q,s)$ converges
  and is indefinitely differentiable.
\end{Lemma}
\begin{proof}
  It is a direct consequence of the proof of the previous lemma,
  using the fact that the successive derivatives of $Q\mapsto Q(x)^{-s}$
  are bounded by a bound proportional to $Q(x)^{-s}$.
\end{proof}

\begin{Definition}\label{defZetaExtreme}
  Let $\quartet{V}{\Q}{Z}{\G}$ be a Voronoi space,
  and let $s_0=s_0(Z)$ as in Lemma~\ref{lemConvergence}.
  Let $Q\in\Q$.
  \begin{enumerate}[(i)]
   \item For $s>s_0$, $Q$ is \emph{\pf{\zeta}-extreme at~$s$} when
   it is a local minimum of the function $Q\mapsto\zeta(Q,s)$.
   \item $Q$ is \emph{finally \pf{\zeta}-extreme} when
   there exists an $s>s_0$ such that $Q$ is \pf{\zeta}-extreme at any $s'>s$.
  \end{enumerate}
\end{Definition}

\subsection{Characterization of final $\zeta$-extremality}
In this section, $\quartet{V}{\Q}{Z}{\G}$ denotes a Voronoi space with $s_0<\infty$,
$Z=-Z$, and $\card{Z}=\infty$.

In order to study the \pf{\zeta}-extremality of a quadratic form~$Q$,
we need to approximate the function $\zeta$ in a neighborhood of $Q$:
\begin{Lemma}\label{lemApproximationZeta}
  Let $s>s_0$, $Q\in\Q$, $H\in\gp_Q$, and $Q_H(x)=Q(x,\exp(H)\,x)$. We have
  \begin{multline*}
    \zeta(Q_H,s) =
      \sum_{x\in Z} Q(x)^{-s}
        \Biggl( 1 - s\frac{Q(x,Hx)}{Q(x)}
          + \frac{s^2}{2}\frac{Q(x,Hx)^2}{Q(x)^2} \\
          - \frac{s}{2}\biggl(\frac{Q(Hx)}{Q(x)}-\frac{Q(x,Hx)^2}{Q(x)^2}\biggr)
          + O\bigl(\norm{H}_Q^3\bigr)
        \Biggr).
  \end{multline*}
\end{Lemma}

\begin{proof}
We have, using $H=H^*_Q$,
\begin{align*}
  Q_H(x)^{-s}
  &= Q\bigl(x, \exp(H)\,x\bigr)^{-s} \\
  &= \Bigl(Q(x) + Q(x,Hx) + \frac12 Q(x,H^2x) + O\bigl(\norm{H}_Q^3\bigr)\Bigr)^{-s} \\
  &= Q(x)^{-s}\biggl(1 + \frac{Q(x,Hx)}{Q(x)} + \frac12\frac{Q(x,H^2x)}{Q(x)} + O\bigl(\norm{H}_Q^3\bigr)\biggr)^{-s}.
\end{align*}
Using $(1+\alpha)^{-s} = 1-s\alpha + \bigl(s(s+1)/2\bigr)\alpha^2 + O(\abs{\alpha}^3)$, we get:
\begin{multline*}
  Q_H(x)^{-s}
    = Q(x)^{-s}\Biggl(1 - s\frac{Q(x,Hx)}{Q(x)} - \frac{s}{2}\frac{Q(x,H^2x)}{Q(x)}  \\[-\jot]
    \shoveright{ + \frac{s(s+1)}{2}\frac{Q(x,Hx)^2}{Q(x)^2}+ O\bigl(\norm{H}_Q^3\bigr)\Biggr) }\\
  \shoveleft{ \hphantom{Q_H(x)^{-s}} = Q(x)^{-s} \Biggl( 1 - s\frac{Q(x,Hx)}{Q(x)} + \frac{s^2}{2}\frac{Q(x,Hx)^2}{Q(x)^2} } \\[-\jot]
    - \frac{s}{2}\biggl(\frac{Q(Hx,Hx)}{Q(x)}-\frac{Q(x,Hx)^2}{Q(x)^2}\biggr)\Biggr) + O\bigl(\norm{H}_Q^3\bigr).
\end{multline*}
To get the claim of the lemma, it remains to sum on $x\in Z$
and to use the absolute convergence of the series.
\end{proof}

Recall that, for $Q\in\Q$, the layer of square radius~$r$ is the set
\begin{equation*}
  Q_r = \{x\in Z \mid Q(x)=r\}.
\end{equation*}
The following result has been proved by Delone and Ryshkov
in the case of classic Voronoi spaces \cite{DeloneRyshkov},
and recently by R.~Coulangeon in the case of Humbert forms
\cite{CoulangeonEpsteinHumbert}.

\begin{Theorem}\label{thmDeloneRyshkov}
  (Compare with Theorem~\ref{thmVoronoi}.)
  Let $\quartet{V}{\Q}{Z}{\G}$ be a Voronoi space such that
  $Z=-Z$ and $s_0<\infty$, and let $Q\in\Q$.
  \begin{enumerate}[(i)]
    \item If $(Q,Q_r)$ is strongly eutactic for every $r>0$ such that $Q_r\ne0$,
      and if $(Q,Q_{\min})$ is perfect,
      then $Q$ is finally \pf{\zeta}-extreme.
    \item If $Q$ is finally \pf{\zeta}-extreme, then $(Q,Q_r)$ is strongly eutactic
    for every $r>0$ such that $Q_r\ne0$, and $(Q,Q_{\min})$ is weakly perfect.
  \end{enumerate}
\end{Theorem}
\begin{proof}
  Let  $r_0=\gamma(Q)<r_1<r_2<\dotsb$ be the square radii of the nonempty layers of $Q$.
  We have, according to Lemma~\ref{lemApproximationZeta},
  \begin{equation*}
    \zeta(Q_H,s) = \zeta(Q,s) + A(H) + B(H) + O\bigl(\norm{H}_Q^3\bigr),
  \end{equation*}
  where $A$ is a linear form on $\gp_Q$ and $B$ is a quadratic form on $\gp_Q$.
  More precisely,
  \begin{gather*}
    A(H) = -s\sum_{i\ge0} r_i^{-s} \sum_{x\in Q_{r_i}} \frac{Q(x,Hx)}{Q(x)}, \\
    B(H) = \sum_{i\ge0}\biggl(\frac{s^2}{2}r_i^{-s} C_i(H) - \frac{s}{2}r_i^{-s}D_i(H)\biggr),
  \end{gather*}
  where
  \begin{gather*}
    C_i(H) = \sum_{x\in Q_{r_i}}\frac{Q(x,Hx)^2}{Q(x)^2}, \\
    D_i(H) = \sum_{x\in Q_{r_i}}\biggl(\frac{Q(Hx)}{Q(x)}-\frac{Q(x,Hx)^2}{Q(x)^2}\biggr).
  \end{gather*}
  Using the definition of the \pf{Q}-norm and the triangular equality, we get
  \begin{equation*}
    0 \le C_i(H) \le \card{Q_{r_i}}\norm{H}_Q^2 \quad \text{and} \quad
    0 \le D_i(H) \le \card{Q_{r_i}}\norm{H}_Q^2.
  \end{equation*}
  To prove the theorem, we use the classic criterion for minima
  on twice continuously differentiable functions:
  \begin{Theorem}
    Let $W$ be a finite-dimensional real vector space, and let $f:W\to\RR$ be a twice continuously differentiable function
    such that
    \begin{equation*}
      f(H) = f(0) + f'_0(H) + f''_0(H) + O\bigl(\norm{H}^3\bigr),
    \end{equation*}
    where $f'_0$ and $f''_0$ are respectively linear and quadratic forms.
    Then, if $f'_0(H)=0$ and $f''_0(H)>0$ for all $H\in W\smz$,
    then $0$ is a local minimum of~$f$.
    Conversely, if $0$ is a (strict) local minimum of~$f$,
    then $f'_0(H)=0$ and $f''_0(H)\ge0$ for all $H\in W$.
  \end{Theorem} 
  Applying the criterion to the function $H\mapsto\zeta(Q_H,s)$, we get:
  \emph{If $A(H)=0$ and $B(H)>0$ for all $H\in\gp_Q\smz$,
    then $Q$ is (strictly) \pf{\zeta}-extreme at~$s$.
    Conversely, if $Q$ is \pf{\zeta}-extreme at~$s$,
    then $A(H)=0$ and $B(H)\ge0$ for all $H\in\gp_Q$.}
  This criterion, together with Claims~A and~B below, achieves the proof.

  \begingroup\medskip\noindent\itshape
    \noindent Claim A. The following are equivalent:
    \begin{enumerate}[(a)]
      \item $A(H)= 0$ for every $H\in\gp_Q$ and for $s$ large enough;
      \item $(Q,Q_{r_i})$ is strongly eutactic for every $i$.
    \end{enumerate}
  \endgroup
  It is clear that (b) implies (a).
  Conversely, suppose that there exists an index~$i$
  such that $(Q,Q_{r_i})$ is not strongly eutactic.
  Let $k$ be the smaller index for which $(Q,Q_{r_k})$
  is not strongly eutactic,
  and let $H\in\gp_Q$ such that $p:=\sum_{x\in Q_{r_i}} Q(x,Hx)/Q(x)\ne0$.
  We have
  \begin{equation*}
    A(H) = -s r_k^{-s} (p+ S),
  \end{equation*}
  where
  \begin{equation*}
    \abs{S}
    \le \biggl|\sum_{i=k+1}^{\infty} \Bigl(\frac{r_{i}}{r_{k}}\Bigr)^{-s}
          \sum_{x\in Q_{r_i}} \frac{Q(x,Hx)}{Q(x)}\biggr|
    \le\norm{H}\sum_{i=k+1}^{\infty}\Bigl(\frac{r_{i}}{r_{k}}\Bigr)^{-s}.
  \end{equation*}
  Now, for $s'=s\bigl(1-\log(r_{k})/\log(r_{k+1})\bigr)$,
  we have $(r_i/r_k)^{-s}\le r_i^{-s'}$ for $i\ge k+1$.
  Therefore,
  \begin{equation*}
    \abs{S} \le \norm{H}_Q\sum_{i\ge0}r_i^{-s'}
      = \norm{H}_Q\zeta(Q,s').
  \end{equation*}
  As $\zeta(Q,s')$ tends to zero when $s$ tends to the infinity,
  for $s$ large enough, we have $\abs{S}\le p/2$,
  and therefore $A(H)\ne0$. This achieves the proof of Claim~A.

  \begingroup\medskip\noindent\itshape
    Claim B. If $(Q,Q_{\min})$ is perfect,
    then $B(H)>0$ for every $H\in\gp_Q\smz$ and for $s$ large enough.
    Conversely, if $B(H)\ge0$ for every $H\in\gp_Q$ and for $s$ large enough,
    and if $(Q,Q_{\min})$ is eutactic,
    then  $(Q,Q_{\min})$ is weakly perfect.
  \endgroup

  \smallskip
  If $(Q,Q_{\min})$ is perfect, then $C_0(H)>0$ when $H\in\gp_Q\smz$.
  The function $f:[H]\mapsto C_0(H)/\norm{H}_Q^2$ defined on $\PP(\gp_Q)$
  is continuous and positive,
  and $\PP(\gp_Q)$ is compact; therefore there exists a $\delta>0$ such that
  $f\ge\delta$, that is
  \begin{equation*}
    C_0(H)\ge\delta\norm{H}_Q^2,\quad \forall H\in\gp_Q.
  \end{equation*}
  Thus, we have
  \begin{gather*}
    B(H) \ge \frac{s^2r_0^{-s}}{2}\Bigl(\delta\norm{H}_Q^2 + S\Bigr), \\
    S = - s^{-1}D_0(H) + \sum_{i=1}^{\infty}\Bigl(\frac{r_i}{r_k}\Bigr)^{-s} \bigl(C_i(H)-s^{-1}D_i(H)\bigr)
  \end{gather*}
  With an argument similar to the one of the proof of Claim~A, it is shown that
  $\abs{S}\le\delta\norm{H}_Q^2/2$ for $s$ large enough,
  and therefore $B(H)>0$ for every $H\in\gp_Q\smz$.

  Conversely, suppose, that $(Q,Q_{\min})$ is eutactic but not weakly perfect.
  Then there exists an $H_0\in\gp_Q$,
  an $x\in Q_{\min}$ and a $c\in\RR$ such that $(H_0-c)x\ne0$
  but $Q(y,(H_0-c)y)=0$
  for every $y\in Q_{\min}$;
  the eutaxy of $(Q,Q_{\min})$ then implies that $c=0$.
  So, we have $C_0(H_0)=0$ and $D_0(H_0)=p>0$. So,
  \begin{gather*}
    B(H_0) = \frac{sr_0^{-s}}{2}\bigl(-p + S\bigr), \\
    S = \sum_{i=1}^{\infty} \Bigl(\frac{r_i}{r_k}\Bigr)^{-s} \bigl(s\,C_i(H_0)-D_i(H_0)\bigr).
  \end{gather*}
  As before, $\abs{S}\le p/2$ for $s$ sufficiently large,
  and therefore $B(H_0)<0$.
  This achieves the proof of Claim~B and of the theorem.
\end{proof}

\subsection{Criterion of $\zeta$-extremality in terms of design}

In this paragraph, we work with alternative Voronoi spaces,
since they are more convenient when talking about designs.
All the definitions and results from the two previous paragraphs
can be formulated in the language of alternative Voronoi spaces.

The following result has been proved by Coulangeon
in the case of the classic alternative Voronoi space
\cite[Theorem~1]{CoulangeonEpstein};
note however that our result generalizes only
partially the theorem of Coulangeon.
Recall that we use the brackets $[\ ]$ to denote the image
through the projection $V\smz\to\PP(V)$.

\begin{Theorem}\label{thmVenkovEpstein}
  (Compare with Corollary~\ref{corPerfectionCriterion}.)
  Let $\quartet{V}{Q}{\Z}{\G}$ be a regular alternative Voronoi space
  such that $Z=-Z$.
  Then there exists a $s_1>0$ with the following property:
  For $Z\in\Z$,
  if $[Z_r]$ is a \pf{4}-design for every~$r\ge0$ such that $Z_r\ne0$,
  then $Z$ is \pf{\zeta}-extreme for every $s>s_1$.
\end{Theorem}
The constant~$s_1$
of the theorem can be computed explicitly by the formula
given in the proof below.
For example in the case of the classic
alternative Voronoi space of rank~$n$,
we have $s_1=n/2$~\cite{CoulangeonEpstein}.

\begin{proof}
  Let $Z\in\Z$ and let $r_0=\gamma(Q)<r_1<r_2<\dotsb$ be the square radii
  of the nonempty layers of $Z$; and suppose that
  $[Z_{r_i}]$ is a \pf{4}-design for every~$i$.
  Recall (Proposition~\ref{propAdjoinizationBis}) that a set in the neighborhood of $Z$
  is of the form $kZ_H$ with $Z_H = \exp(H/2)Z$, $k\in\GK$, and $H\in\gp$ with $\norm{H}_Q$ small.
  From Lemma~\ref{lemApproximationZeta}, we have
  \begin{equation*}
    \zeta(Z_H,s) = \zeta(Z,s)
      + \sum_{i\ge0} r_i^{-s} \Bigl( -s\,A_i(H) + \frac{s^2}{2} B_i(H)
          - \frac{s}{2}\,C_i(H) + O\bigl(\norm{H}_Q^3\bigr) \Bigr),
  \end{equation*}
  where
  \begin{gather*}
    A_i(H) = \sum_{x\in Z_{r_i}}\frac{Q(x,Hx)}{Q(x)}, \qquad
    B_i(H) = \sum_{x\in Z_{r_i}}\frac{Q(x,Hx)^2}{Q(x)^2}, \\
    C_i(H)= \sum_{x\in Z_{r_i}} \biggl(\frac{Q(Hx)}{Q(x)}-\frac{Q(x,Hx)^2}{Q(x)^2}\biggr).
  \end{gather*}
  As $[Z_{r_i}]$ is in particular a \pf{\{2\}}-design,
  we have $A_i(H)=0$ by Lemma~\ref{lemAverageQuadraticFunction}. Let
  \begin{equation*}
    \beta_H(x) = \frac{Q(x,Hx)^2}{Q(x)^2} \ge0, \qquad
    \gamma_H(x) = \frac{Q(Hx)}{Q(x)}-\frac{Q(x,Hx)^2}{Q(x)^2} \ge0.
  \end{equation*}
  We have $\beta_H\in\Fc_{\{2,2\}}\subseteq\Fc_4$ and $\gamma_H\in\Fc_{\{2,2\}}+\Fc_{\{4\}}=\Fc_4$.
  Since $[Z_{r_i}]$ is a \pf{4}-design, we have $B_i(H)=\card{Z_{r_i}}\ave{\beta_H}$
  and $C_i(H)=\card{Z_{r_i}}\ave{\gamma_H}$, so
  \begin{align*}
    \zeta(Z_H,s)
     &= \zeta(Z,s) + \sum_{i\ge0} \card{Z_{r_i}} r_i^{-s}
      \Bigl( \frac{s^2}{2}\ave{\beta_H} - \frac{s}{2} \ave{\gamma_H} + O\bigl(\norm{H}_Q^3\bigr) \Bigr), \\
     &= \zeta(Z,s) \Bigl(1+ \frac{s^2}{2} \ave{\beta_H} 
        - \frac{s}{2} \ave{\gamma_H} + O\bigl(\norm{H}_Q^3\bigr) \Bigr).
  \end{align*}
  Now consider the function
  \begin{math}f:\PP(\gp)\to\mathopen]0,\infty\mathclose[\end{math}
  defined by $f([H]) = \ave{\gamma_H}/\ave{\beta_H}$,
  and let $s_1$ be the maximal value of~$f$.
  For $s>s_1$, we have $\ave{\gamma_H}\le s_1\ave{\beta_H}\le s\norm{H}_Q^2$, hence
  \begin{equation*}
    \zeta(Z_H,s)
      \ge \zeta(Z,s) \Bigl(1+ \frac{s(s-s_1)}{2}\norm{H}_Q^2 + O\bigl(\norm{H}_Q^3\bigr) \Bigr).
  \end{equation*}
  We infer that $\zeta(kZ_H,s)=\zeta(Z_H,s)>\zeta(Z,s)$ for $\norm{H}_Q$ sufficiently small.
  Note that the value of $s_1$ depends only on the structure of alternative Voronoi space of $\quartet{V}{Q}{\Z}{\G}$.
\end{proof}

As for standard extremality, we can exploit
the symmetry of the potentially extreme sets
in order to check that there are indeed extreme.
So, using Corollary~\ref{corInvarianceDesign},
we draw the following consequence:

\begin{Corollary}\label{corZetaInvarianceCriterion}
  (Compare with Corollary~\ref{corInvarianceCriterion}.)
  Let $\quartet{V}{Q}{\Z}{\G}$ be a regular alternative Voronoi space,
  and let $Z\in\Z$. Let $\mathbf{F}=\{g\in \GK\mid gZ=Z\}$.
  If the only \pf{\mathbf{F}}-invariant functions in $\Fc_4$
  are the constant functions,
  then $Z$ is \pf{\zeta}-extreme at every $s>s_1$.
\end{Corollary}

For example, this criterion allows to show that
all the lattices given at the end
of Paragraph~\ref{parDesignInvariant},
are \pf{\zeta}-extreme at any $s>n/2$.

\raggedright


\begin{thebibliography}{VenMar01}

  \bibitem[Bach05]{Bachoc}
    Ch.~Bachoc,
    \emph{Designs, groups and lattices},
    J. Th\'eor. Nombres Bordeaux \textbf{17} (1) (2005) 25--44.
    
  \bibitem[BaBaCo04]{DesignGrassmannian2}
    Ch.~Bachoc, E.~Bannai, R.~Coulangeon,
    \emph{Codes and designs in Grassmannian spaces},
    Discrete Math. \textbf{277} (1--3) (2004) 15--28.

  \bibitem[BaCoNe02]{DesignGrassmannian}
    Ch.~Bachoc, R.~Coulangeon, G.~Nebe,
    \emph{Designs in Grassmannian spaces and lattices},
    J. Algebraic Combin. \textbf{16} (1) (2002) 5--19.
    
  \bibitem[BaeIca97]{BaezaIcaza}
    R.~Baeza, M.~I.~Icaza,
    \emph{On Humbert-Minkowski's constant for a number field},
    Proc.\ Amer.\ Math.\ Soc.\ \textbf{125} (11) (1997) 3195--3292.

  \bibitem[Bava97]{Bavard}
    Ch.~Bavard,
    \emph{Systole et invariant d'Hermite},
    J. Reine Angew. Math. \textbf{482} (1997) 93--120.
    
  \bibitem[Bava01]{BavardSymplectique}
    Ch.~Bavard,
    \emph{Familles hyperboliques de r\'eseaux symplectiques},
    Math. Ann. \textbf{320} (4) (2001) 799--833.
  
  \bibitem[Bava07]{BavardLorentzien}
    Ch.~Bavard,
    \emph{Invariant d'Hermite isotrope et densit\'e des r\'eseaux orthogonaux lorentziens},
    Comment. Math. Helv. \textbf{82} (2007) 39--60.

  \bibitem[BerMar89]{DualExtreme}
    A.-M.~Berg\'e, J.~Martinet,
    \emph{Sur un probl\`eme de dualit\'e li\'e aux sph\`eres en g\'eom\'etrie des nombres},
    J. Number Theory \textbf{32} (1) (1989) 14--42.

  \bibitem[BerMar91]{ExtremeAutomorphism}
    A.-M.~Berg\'e, J.~Martinet,
    \emph{R\'eseaux extr\^emes pour un groupe d'automorphismes},
    Ast\'erisque \textbf{198--200} (1991) 41--66.

  \bibitem[BerMar95]{LatticeFamily}
    A.-M.~Berg\'e, J.~Martinet,
    \emph{Densit\'e dans des familles de r\'eseaux. Application aux r\'eseaux isoduaux},
    Enseign. Math. \textbf{41} (1995) 335--365.
    
  \bibitem[BusSar94]{BuserSarnak}
    P.~Buser, P.~Sarnak,
    \emph{On the period matrix of a Riemann surface of large genus}
    (with appendix by J.~H.\ Conway and N.~J.~A.\ Sloane),
    Invent. Math. \textbf{117} (1994) 27--56.

  \bibitem[Coul96]{CoulangeonkExtreme}
    R.~Coulangeon,
    \emph{R\'eseaux $k$-extr\^emes},
    Proc. London Math. Soc. \textbf{73} (3) (1996) 555--574.

  \bibitem[Coul01]{CoulangeonVenkov}
    R.~Coulangeon,
    \emph{Vorono\"\i\ theory over algebraic number fields},
    chap.~5 of~\cite{MartinetVenkov}.

  \bibitem[Coul06]{CoulangeonEpstein}
    R.~Coulangeon,
    \emph{Spherical designs and zeta functions of lattices},
    Int. Math. Res. Not. (2006) Art. ID~49620.

  \bibitem[Coul08]{CoulangeonEpsteinHumbert}
    R.~Coulangeon,
    \emph{On Epstein's zeta function of Humbert forms},
    to appear in: Int. J. Number Theory.

  \bibitem[DelRy\v{s}67]{DeloneRyshkov}
    B.~N.~Delone, S.~S.~Ry\v skov,
    \emph{A contribution to the theory of the extrema of a multi-dimensional $\zeta$\nobreak-function},
    Dokl. Akad. Nauk SSSR \textbf{173}, 991--994 (Russian),
    Soviet Math. Dokl. \textbf{8} (1967) 499--503.

  \bibitem[DeGoSe77]{DelsarteGoethalsSeidel}
    P.~Delsarte, J.-M.~Goethals, J.~J.~Seidel,
    \emph{Spherical codes and designs},
    Geometriae Dedicata  \textbf{6} (1977) 363--388.

  \bibitem[Gods93]{Godsil}
    C.D.~Godsil,
    \emph{Algebraic combinatorics},
    Chapman \& Hall (New York, London) 1993.

  \bibitem[GoeSei79]{GoethalsSeidel}
    J.-M.~Goethals, J.~J.~Seidel,
    \emph{Spherical Designs},
    Proc.\ Sympos.\ Pure Math., XXXIV, 255--272,
    Amer.\ Math.\ Soc. (Providence R.I.) 1979.

  \bibitem[HarPac05]{Gaeta}
    P.~de la Harpe, C.~Pache,
    \emph{Cubature formulas, geometrical designs, reproducing kernels, and Markov operators},
    in: \emph{Infinite Groups: Geaometric, Combinatorial and Dynamical Aspects},
    Birkh\"auser (Basel, Boston, Berlin) (2005) 219--267.

  \bibitem[Icaz97]{Icaza}
    M.~I.~Icaza,
    \emph{Hermite constant and extreme forms for algebraic number fields},
    J. London Math. Soc. (2) \textbf{55} (1) (1997) 11--22.

  \bibitem[Mart03]{MartinetBook}
    J.~Martinet,
    \emph{Perfect lattices in Euclidean spaces},
    Grundlehren der Mathematischen Wissenschaften 327. Springer-Verlag  (Berlin) 2003.

  \bibitem[MartV01]{MartinetVenkov}
    J.~Martinet, ed.,
    \emph{R\'eseaux euclidiens, designs sph\'eriques et formes modulaires,
    Autour des travaux de B.~Venkov},
    L'En\-sei\-gne\-ment Ma\-th\'e\-ma\-tique,
    monographie n$\rm^o$~37
    (Gen\`eve) 2001.

  \bibitem[OniVin94]{OV}
    A.~L.~Onishchik, B.~Vinberg, eds.,
    \emph{Lie groups and Lie algebras III},
    Encyclopaedia of Mathematical Sciences 41, Springer (Berlin) 1994.
  
  \bibitem[Pach06]{PacheThesis}
    C.~Pache,
    \emph{Espaces polynomiaux et formules de cubature},
    doctoral thesis (Gen\`eve) 2006. Available at:
    \hbox{\small\texttt{http://www.unige.ch/cyberdocuments/theses2006/PacheC/meta.html}}.
    
  \bibitem[Rank53]{Rankin}
    R.~A.~Rankin,
    \emph{On positive definite quadratic forms},
    J.\ London Math.\ Soc.\ \textbf{28} (1953) 309--314.

  \bibitem[Sobo62]{Sobolev}
    S.~L. Sobolev, \emph{Cubature formulas on the sphere which are invariant under
    transformations of finite rotation groups}, Dokl.\ Akad.\ Nauk SSSR
    \textbf{146} (1962), 310--313 (Russian).

  \bibitem[Sobo96]{SobolevBook}
    S.~L.~Sobolev, V.~L.~Vaskevich,
    \emph{\textcyrillic{Kurbaturnye formuly}},
    edited and with a contribution by M.~D.~Ramazanov,
    Izdatel'stvo Rossi\u\i sko\u\i\ Akademii Nauk, Sibirskoe Otdelenie, Institut Matematiki im. S.~L. Soboleva
    (Novosibirsk) 1996 \\
    =~\emph{The theory of cubature formulas},
    translated from the 1996 Russian original and with a foreword by S.~S.~Kutateladze,
    Kluwer Academic Publishers Group (Dordrecht) 1997
    (original in Russian).

  \bibitem[VenMar01]{VenkovMartinet1}
    B.~B.~Venkov (notes by J.~Martinet),
    \emph{R\'eseaux et designs sph\'eriques},
    Chap.~1 of \cite{MartinetVenkov}.

  \bibitem[Voro08]{Voronoi}
    G.~Voronoi,
    \emph{Nouvelles applications des param\`etres continus \`a la th\'eorie
    des formes quadratiques~:
    1.\ Sur quelques propri\'et\'es des formes quadratiques positives parfaites},
    J.\ reine angew.\ Math.\ \textbf{133} (1908) 97--178.

\end{thebibliography}
\end{document}